\documentclass[a4paper]{amsart}
\usepackage{graphicx, url, color, xcolor, amssymb, amsmath,enumitem,multirow,appendix, bbm} 

\usepackage{hyperref}

\theoremstyle{plain}
\newtheorem{thm}{Theorem}[section]
\newtheorem*{thm*}{Theorem}
\newtheorem{lem}[thm]{Lemma}

\newtheorem{prop}[thm]{Proposition}

\newtheorem{innercustomthm}{Theorem}
\newenvironment{customthm}[1]
  {\renewcommand\theinnercustomthm{#1}\innercustomthm}
  {\endinnercustomthm}

\theoremstyle{definition}

\newtheorem{example}[thm]{Example}

\title{Hecke modules associated to certain even ternary lattices of half-discriminant $\mathbf{N^2}$}
\author{Samadrito Bhattacharyya, Giovanna De Lauri, Dan Fretwell}
\date{}

\begin{document}

\begin{abstract}
We consider Hecke modules of orthogonal modular forms attached to definite even ternary lattices of half-discriminant $N^2$, where $N > 1$ is odd and square-free. We prove explicit Jacquet-Langlands style correspondences with spaces of classical modular forms (in the spirit of Birch \cite{birch1991hecke} and Hein-Tornaria-Voight \cite{heintornariavoight}). This is done by applying sequences of Watson-style transformations that reduce the discriminant while preserving the adelic orthogonal group.
\end{abstract}

\maketitle

\section{Introduction}

Let $\Lambda\subset V$ be a lattice in a ternary rational quadratic space $(V,Q)$. Since the celebrated Hasse-Minkowski theorem tells us that $(V,Q)$ is determined up to isometry by its completions, one naturally wonders whether the same is true for $\Lambda$ (i.e.\ if $\Pi\subset V$ is locally isometric to $\Lambda$ everywhere, then does it follow that it is globally isometric to $\Lambda$). Unfortunately, it has been known for a long time that the answer is in general no, with the obstruction being measured by the class set of $\Lambda$, the finite set $\text{Cl}(\Lambda) = O(V)\backslash \text{gen}(\Lambda)$. The most interesting case is when $(V,Q)$ is positive definite, as the class set is very well understood otherwise.

Seminal work of Kneser \cite{Kneser1957} shows that every isometry class in $\text{Cl}(\Lambda)$ can be computed by beginning with the lattice $\Lambda$ and computing iterated $p_i$-neighbours for some finite sequence of primes $p_i$ (a $p$-neighbour of $\Lambda$ is a lattice $\Pi\subset V$ with the same scale as $\Lambda$ and satisfying $\Lambda/(\Lambda\cap\Pi) \cong \Pi/(\Lambda\cap\Pi)\cong\mathbb{Z}/p\mathbb{Z}$). The set of $p$-neighbours of a lattice is a finite set and is explicitly computable. If there is only one spinor genus in $\text{gen}(\Lambda)$ then in fact one can apply Kneser's result using a single fixed prime $p\nmid 2\,\text{disc}(\Lambda)$.

Naturally, one wonders how the $p$-neighbours of a lattice $\Pi\in\text{gen}(\Lambda)$ distribute among the isometry classes in $\text{Cl}(\Lambda)$? Alternatively, consider the directed graphs $G_p$ with vertices corresponding to $\text{Cl}(\Lambda) = \{[\Lambda_1],[\Lambda_2],...,[\Lambda_h]\}$ and directed edges from $[\Lambda_i]$ to $[\Lambda_j]$ for each $p$-neighbour of $\Lambda_i$ that is isometric to $\Lambda_j$. Quite often $G_p$ is connected (by Kneser's result, when $p$-neighbours with respect to the prime $p$ suffice to exhaust the class set) and one wishes to study the entries of the associated adjacency matrices $T_p$. Chenevier \cite{chenevier2021statistics} recently considered this question, showing precise distribution results for large $p$. 

Going further, one might wonder about the spectral properties of the directed graphs $G_p$. Originally, this was considered by Birch, who observed that often the non-trivial eigenvalues of the family of matrices $T_p$ match Hecke eigenvalues of weight $2$ modular forms (see \cite{birch1991hecke}, also discussed in \cite{hein2016orthogonal} and \cite{heintornariavoight}). It is this connection that shows that for such lattices the family of directed graphs $G_p$ are Ramanujan graphs.

Alternatively, $p$-neighbour data can be studied by instead dualising and looking at the space of functions: \[\mathcal{M} = \{f:\text{Cl}(\Lambda) \rightarrow\mathbb{C}\}\cong\mathbb{C}^h.\]  By averaging such functions over $p$-neighbours of the input lattice, one gets an operator $T_p$. These operators are then the adjacency matrices considered above and so have the same eigenvalues. Such a space of functions is a special example of a space of (definite) orthogonal modular forms, and the operators $T_p$ are examples of Hecke operators acting on this space (a special example of a Hecke-module of algebraic modular forms, due to Gross \cite{gross1999algebraic}).

A natural generalisation of this idea is to consider more general spaces of (definite) orthogonal modular forms attached to $\Lambda$: \[\mathcal{M}(O(V),O(\widehat{\Lambda}), W) = \{f: \text{gen}(\Lambda) \rightarrow W\,|\, f(\gamma \Pi) = \gamma\cdot f(\Pi) \text{ for all } \gamma\in O(V)\}\] with $W$ is a finite-dimensional representation of $O(V)$, called the weight. The trivial representation $W = \mathbb{C}$ then gives rise to the space of $O(V)$-invariant functions, hence is the space $\mathcal{M}$ above. As before, there are Hecke operators $T_p$ acting on these spaces, via averaging over $p$-neighbours. The insights of Birch are immediately interpreted as providing explicit transfer results between Hecke-modules of orthogonal modular forms and classical modular forms, in the spirit of the Jacquet-Langlands correspondence (a special case of Langlands Functoriality).

The recent paper of Hein-Tornaria-Voight \cite{heintornariavoight} generalises the observations of Birch and proves a Hecke-equivariant transfer result for a wide range of lattices with non-square half-discriminant. Roughly speaking, if $\Lambda\subset V$ is an even lattice that corresponds via the even Clifford functor to a residually unramified order $\mathcal{O}\subset B$ in a definite quaternion algebra $B$, then they are able to prove a Hecke-equivariant isomorphism between spaces of orthogonal modular forms attached to $\Lambda$ and spaces of classical modular forms. The strength of these results is twofold: 

\begin{itemize}
\item{they treat spaces of orthogonal modular forms of varying weight, showing that other classical modular forms can be computed this way (e.g.\ not just weight $2$),}
\item{the results are entirely explicit, giving a precise and complete description of the image of the transfer in terms of weight, level, newness and Atkin-Lehner eigenvalues (at the bad primes).}
\end{itemize}

More recently, a paper of De Lauri-Fretwell \cite{delaurifretwell} began the exploration of genera corresponding to lattices with square half-discriminant, i.e.\ $p^2$ with $p$ prime and $4p^2$ with $p$ an odd prime. For each genus they were able to either prove a Hecke-equivariant transfer result or conjecture one (the difficulty being that the corresponding orders are residually ramified at either $2$ or $p$).

In this paper we continue to explore the square half-discriminant case, considering the case of even lattices of half-discriminant $N^2$ with $N$ odd and square-free. For every such genus we are able to prove that there exists a Hecke-equivariant isomorphism with a Hecke-module of classical modular forms. Our results are explicit, describing the image of the transfer as precisely as possible (as with the results in \cite{heintornariavoight} and \cite{delaurifretwell}). The following is our main theorem (with notation as in Sections $2$-$4$).

\begin{thm}\label{main}
Let $\Lambda$ be an even lattice in a positive definite ternary quadratic space $(V,Q)$ over $\mathbb{Q}$. Let $\mathrm{disc}(\Lambda) = 2N^2$ with $N$ odd and square-free and suppose that $N = N_1N_2N_3N_4$ according to Subsection \ref{N1N2N3N4}. For each $k\geq 0$ and $d \,|\, N$ there exists a Hecke-equivariant isomorphism (for Hecke operators away from $N$): \[\mathcal{S}(O(V),O(\widehat{\Lambda}),W_{2k}\otimes \theta_d) \cong \bigoplus_{M_4|N_4}S_{2k+2}^{N_1M_4^2-\textup{new}}(\Gamma_0(N_1N_2N_3^2M_4^2), \varepsilon_{N_1}^{(N_1N_2N_3^2M_4^2)}\varepsilon_d^{(N_1N_2N_3^2M_4^2)})\]
\end{thm}

As previously mentioned, such theorems cannot immediately be proved using the results of Hein-Tornaria-Voight since the associated quaternion orders are residually ramified at at least one prime. Our general strategy in proving such a result is to instead perform certain Watson-style transformations to $\Lambda$, producing a new even lattice $\Pi$ of half-discriminant $N'$ for some $N'|N$ while preserving the adelic orthogonal group (see \cite{watson1953some}, \cite{watson1963class}). This is enough to prove a bijection between $\text{Cl}(\Lambda)\rightarrow\text{Cl}(\Pi)$ that respects the $p$-neighbour relation, hence providing a Hecke-equivariant isomorphism between the associated Hecke-modules of orthogonal modular forms. One can then appeal to the above mentioned result of Hein-Tornaria-Voight to get the required Hecke-equivariant transfer to classical spaces of modular forms. 

In terms of spaces of orthogonal modular forms, one can view the application of such transformations as lowering the level of the forms, hence also shedding light on the newform theory for spaces of orthogonal modular forms. In particular, the theorem above can often be interpreted as saying that various spaces of orthogonal modular forms attached to $\Lambda$ consist entirely of oldforms.

\section*{Acknowledgements}
This work is the product of a collaboration between the authors. The first named author became interested in this project while undertaking his second year of undergraduate studies, and was funded by an EPSRC Vacation Internship. The second named author became interested in this project during her PhD studies. Both were supervised by the third named author. We thank both EPSRC and Lancaster University for making the internship and the PhD a possibility.

\section{Setup and Notation}

We will work with a similar setup to the paper \cite{delaurifretwell}, but it will be fruitful for the reader for us to fully recap the necessary background and notation used there.

\subsection{Quadratic Spaces and Lattices}

Let $(V,Q)$ be a quadratic space over a field $F$ with $\text{char}(F)\neq 2$, i.e.\ an $F$-vector space $V$ equipped with a map $Q: V \rightarrow F$ such that: \begin{itemize}
\item{$Q(\alpha v) = \alpha^2 Q(v)$ for all $\alpha\in F$ and $v\in V$,}
\item{the map $T(v,w) := Q(v+w) - Q(v) - Q(w)$ is bilinear.}
\end{itemize} Note that as a consequence of our definition we relate the quadratic map and the symmetric bilinear form via $T(v,v) = 2Q(v)$. We will also assume from now on that $n = \text{dim}_{\mathbb{Q}}(V) < \infty$.

Given any basis $\{e_1,e_2,...,e_n\}$ for $V$ we define the corresponding Gram matrix $T = (T(e_i,e_j))_{1\leq i, j\leq n}$ and the discriminant $\mathrm{disc}(V) = [\text{det}(T)]\in F/F^{\times 2}$. We say that $(V,Q)$ is non-degenerate if $\mathrm{disc}(V) \neq 0$.

An isometry between quadratic spaces $(V_1,Q_1)$ and $(V_2,Q_2)$ over $F$ is an isomorphism $\phi: V_1\rightarrow V_2$ such that $Q_2(\phi(v)) = Q_1(v)$ for all $v\in V_1$. If an isometry exists we will write $(V_1,Q_1)\sim (V_2,Q_2)$. We will also write $(V,Q)\sim \langle A_1,A_2,...,A_m\rangle$ to signify that $(V,Q)$ is isometric to a quadratic space whose Gram matrix under some basis is block diagonal with blocks $A_i$ (with lower case $a_i$ being reserved for the special case of a $1\times 1$ block). The dimension $n$ and the discriminant $\mathrm{disc}(V)$ are isometry invariants of the quadratic space.

It is very well known that every quadratic space $(V,Q)$ over $F$ can be diagonalised, so that $(V,Q)\sim \langle a_1,a_2,...,a_n\rangle$ for some $a_i\in F$. The discriminant is then given by the product $\mathrm{disc}(V) = [a_1a_2...a_n]$.

The set of self-isometries of a quadratic space $(V,Q)$ over $F$ gives the orthogonal group of $V$: \[O(V) = \{\phi\in\text{GL}(V)\,|\, Q(\phi(v)) = Q(v)\}.\] When $(V,Q)$ is non-degenerate every element of $O(V)$ has determinant $\pm 1$ and the kernel of the determinant map is the special orthogonal group $SO(V)$. This gives an isomorphism $O(V)/SO(V) \cong \{\pm 1\}$ and when $n$ is odd the non-trivial coset is represented by $-I$.

If $(V,Q)$ is a quadratic space over $F$ and $\lambda\in F$ then we may scale the quadratic map by $\lambda$ to get the quadratic space $(V, Q^{(\lambda)})$ with quadratic map $Q^{(\lambda)}:  V \to F$ given by $Q^{(\lambda)}(v) := \lambda Q(v)$ for all $v\in V$. The associated symmetric bilinear form is then $T^{(\lambda)}(v,w) = \lambda T(v,w)$. If $T$ is a Gram matrix for $(V,Q)$ under some basis then the Gram matrix of $(V,Q^{(\lambda)})$ under the same basis is $\lambda T$. Note that this notion of scaling is not the same as scaling a basis of $V$ by $\lambda$, which would produce a quadratic space with Gram matrix $\lambda^2 T$.

From now on let $(V,Q)$ be a non-degenerate quadratic space over $\mathbb{Q}$ of dimension $n$. Attached to $V$ are various completed quadratic spaces $(V_{\infty},Q)$ and $(V_p,Q)$ for $p$ prime (these being quadratic spaces over $\mathbb{R}$ and $\mathbb{Q}_p$ respectively). The underlying vector spaces are: \begin{align*}V_{\infty} &= V\otimes_{\mathbb{Q}}\mathbb{R}\\ V_p &= V\otimes_{\mathbb{Q}}\mathbb{Q}_p \end{align*} and the quadratic map $Q$ naturally extends to these spaces via $Q(v\otimes \alpha) = \alpha^2Q(v)$. 

The notion of isometry extends naturally to local quadratic spaces, and the celebrated Hasse-Minkowski theorem (Chapter $6$ of \cite{cassels2008rational}) tells us that the global isometry class of $(V,Q)$ is determined by the local isometry classes of $(V_{\infty},Q)$ and $(V_p,Q)$. The classification of local quadratic spaces is known, and consequently the global isometry class of $(V,Q)$ is fully determined by the following list of invariants:

\begin{itemize}
\item{The discriminant $\mathrm{disc}(V)\in\mathbb{Q}^{\times}/\mathbb{Q}^{\times 2}$,}
\item{The signature $(r,s)$ of $(V_{\infty},Q)$, the unique integers $r,s\geq 0$ such that: \[(V_{\infty},Q) \sim \langle 1,1,...,1,-1,-1,...,-1\rangle,\] with $r$ occurrences of $1$ and $s$ occurrences of $-1$ (noting that $n = r+s$).}
\item{The Hasse invariant $S(V_p)\in\{\pm 1\}$ of $(V_p,Q)$, given by: \[S(V_p) = \prod_{1\leq i<j\leq n}(a_i,a_j)_p,\] where \[(V_p,Q) \sim \langle a_1,a_2,...,a_n\rangle\] is any diagonalisation and $(a,b)_p$ is the Hilbert symbol (when $n = 1$ we declare $S(V_p) = 1$).}
\end{itemize} The local Hasse invariants satisfy certain compatibility conditions (the second of which follows from Hilbert Reciprocity): 

\begin{itemize}
\item{$S(V_p) = 1$ for all but finitely many $p$}
\item{$\prod_{p\leq \infty}S(V_p) = 1$ (where $S(V_{\infty}) = (-1)^{\frac{s(s-1)}{2}}$).}
\end{itemize}

Conversely, suppose these invariants are chosen freely to satisfy the compatibility conditions. Then there exists a rational quadratic space giving rise to these invariants, with only two exceptions: 

\begin{itemize}
\item{$n=1$ and $S(V_p) = -1$ for some finite prime $p$,}
\item{$n=2$, $\mathrm{disc}(V_p) = [-1]$ and $S(V_p) = -1$ for some finite prime $p$.}
\end{itemize}

The completed spaces come equipped with local orthogonal groups: \begin{align*}O(V_{\infty}) &= \{\phi\in\text{GL}(V_{\infty})\,|\, Q(\phi(v)) = Q(v)\}\\ O(V_p) &= \{\phi\in\text{GL}(V_p)\,|\, Q(\phi(v)) = Q(v)\}\end{align*}

If $V$ is positive definite, i.e.\ it has signature $(r,s) = (n,0)$, then $O(V_{\infty}) \cong O_n(\mathbb{R})$ is the standard compact orthogonal group.

A lattice $\Lambda\subset V$ is a free $\mathbb{Z}$-module of rank $n$. Every basis of $\Lambda$ gives rise to a basis for $V$ and a Gram matrix $T$. The discriminant of $\Lambda$ is then $\mathrm{disc}(\Lambda) = [\text{det}(T)]\in\mathbb{Q}^{\times}/(\mathbb{Z}^{\times})^2 = \mathbb{Q}^{\times}$.

If $\Lambda_1\subset V_1$ and $\Lambda_2\subset V_2$ are lattices then an isometry is an isometry $\phi: V_1 \rightarrow V_2$ such that $\phi(\Lambda_1) = \Lambda_2$. The rank $n$ and discriminant $\mathrm{disc}(\Lambda)$ are both isometry invariants of a lattice. The self-isometries of $\Lambda$ give the orthogonal group of $\Lambda$: \[O(\Lambda) = \{\phi\in O(V)\,|\, \phi(\Lambda) = \Lambda\}.\]

If $\Lambda$ satisfies $T(x, y) \in \mathbb{Z}$ for all $x, y \in \Lambda$ then we call $\Lambda$ integral, and if further we have that the scale ideal is equal to $\mathbb{Z}$, i.e.\ $\mathfrak{s}(\Lambda) = \langle T(x,y)\,|\, x,y\in\Lambda \rangle = \mathbb{Z}$ then it is primitive. A primitive lattice is even if $Q(x) \in  \mathbb{Z}$ for all $x \in \Lambda$, otherwise it is odd. We define the norm ideal of $\Lambda$ to be the ideal $\mathfrak{n}(\Lambda) = \langle Q(x)\,|\, x\in\Lambda\rangle \subset\mathbb{Q}$.

One defines the discriminant group of $\Lambda$ to be $\Delta(\Lambda) = \Lambda^{\ast}/\Lambda$, where the dual lattice of $\Lambda$ with respect to $T$ is $\Lambda^{\ast}:= \{v \in V \,:\, T(v, \Lambda) \subseteq \mathbb{Z}\}$. Whenever $\Delta(\Lambda)$ is a direct sum of cyclic groups of prime order for an integral lattice, then the lattice is called elementary.

As with quadratic spaces, lattices can be completed to give local lattices for $p$ prime: \begin{align*}\Lambda_{\infty} &= \Lambda\otimes_{\mathbb{Z}}\mathbb{R} = V_{\infty}\\ \Lambda_p &= \Lambda\otimes_{\mathbb{Z}}\mathbb{Z}_p\subset V_p\end{align*} One might hope that this collection of local lattices completely determines $\Lambda$ up to isometry, but the analogue of the Hasse-Minkowski theorem fails. To quantify the failure, one considers the genus and class set: \begin{align*}\text{gen}(\Lambda) &= \{\Pi\subset V\,|\, \Pi_p \sim\Lambda_p \text{ for all } p\}\\ \text{Cl}(\Lambda) &= O(V)\backslash\text{gen}(\Lambda)\end{align*} The class set is known to be a finite set by a geometry of numbers argument, and its size $h = |\text{Cl}(\Lambda)|$ is the class number of $\Lambda$.

We can still define the notions of dual lattice, discriminant group, integral lattice, primitive lattice, scale ideal and norm ideal in the local lattice case by replacing $\mathbb{Z}$ with $\mathbb{Z}_p$. The notion of even lattice can also be defined over $\mathbb{Z}_2$ in the obvious way.

Since one direction of Hasse-Minkowski for lattices holds (i.e.\ global to local), it's still useful to have a classification of local lattices up to isometry (treated for example in Gerstein \cite{gerstein2008basic}). Indeed, for every prime $p$ it is known that $\Lambda_p$ has a Jordan decomposition, i.e.\ $\Lambda_p = \Lambda_{p,1} \perp \dots \perp  \Lambda_{p,k}$ where each $\Lambda_{p,i}$ has rank $n_i$ and is $p^{a_i}$-modular for some $a_i\in\mathbb{Z}$, i.e.\ has Gram matrix $p^{a_i} U_i$ for some $U_i\in\text{GL}_{n_i}(\mathbb{Z}_p)$. We usually rearrange the $\Lambda_{p,i}$ so that $a_1 < a_2 < ... < a_k$ and refer to the case $a_i = 0$ as being a unimodular lattice.

When $p$ is odd it is known that $\Lambda_{p,i}\sim p^{a_i}\langle 1,1,...,1,\delta_i\rangle$ for a unique $\delta_i\in\{1,r\}$, where $r$ is a quadratic non-residue mod $p$ (we will use this notation later). The collection of invariants $(n_i, a_i, \delta_i)_i$ completely determines $\Lambda_p$ up to isometry and all possibilities can occur (completing the classification of local lattices in this case). We will often use Conway-Sloane $p$-adic symbols to denote isometry classes of local lattices (see \cite[Chapter 15]{conway2013sphere} or \cite{allcock2020conway} for details). 

When $p=2$ it is not necessarily true that $\Lambda_{2,i}$ can be diagonalised. However it is known that: \[\Lambda_{2,i} \sim\begin{cases} 2^{a_i}\langle \delta_1,\delta_2,...,\delta_n\rangle & \text{ if } \mathfrak{n}(\Lambda_{2,i}) = 2^{a_i-1}\mathbb{Z}_2\\ 2^{a_i}\left\langle\begin{pmatrix}0 & 1\\ 1 & 0\end{pmatrix},\begin{pmatrix}0 & 1\\ 1 & 0\end{pmatrix},...,\begin{pmatrix}0 & 1\\ 1 & 0\end{pmatrix},A\right\rangle & \text{ if } \mathfrak{n}(\Lambda_{2,i}) = 2^{a_i}\mathbb{Z}_2 \end{cases},\] where each $\delta_i \in\{\pm 1,\pm 5\}$ and $A\in\left\{\begin{pmatrix}0 & 1\\ 1 & 0\end{pmatrix},\begin{pmatrix}2 & 1\\ 1 & 2\end{pmatrix}\right\}$. The lattices of the first kind are often referred to as Type I lattices and the second as Type II lattices. 

Unfortunately, the Jordan decomposition is not enough to classify $2$-adic lattices up to isometry. However, it is still possible to complete the classification by giving further equivalences, for example by the result of O'Meara (e.g.\ Theorem $8.11$ of \cite{gerstein2008basic}). We will often use Conway-Sloane $2$-adic symbols to denote isometry classes of local lattices up to isometry, and assume the result that two $2$-adic lattices are isometric if and only if their symbols are equal or related by a sequence of Oddity Fusion and Sign Walking rules (see \cite[Chapter 15]{conway2013sphere} or \cite{allcock2020conway} for details).

As a consequence of the classification of local lattices we find that when $\Lambda, \Pi \subset V$ are integral lattices with $ \mathrm{disc}(\Lambda) = \mathrm{disc}(\Pi)$, then $\Lambda_p \simeq \Pi_p$ for all $p \nmid 2\,\mathrm{disc}(\Lambda)$ (since the local lattice must be diagonalisable and unimodular, hence completely determined by its discriminant).

As with quadratic spaces, we often wish to know when a collection of local lattices $\{\Lambda_p\}_{p\leq \infty}$ arises from a global lattice $\Lambda$. By Lemma $5.1$ of Chapter $9$ of Cassels \cite{cassels2008rational}, $\Lambda$ will exist if and only if the corresponding quadratic space $(V,Q)$ in which it embeds exists (relative to the collection of local quadratic spaces $(V_p,Q)$ for each $p\leq \infty$). This is simply a case of checking that the global compatibility conditions mentioned earlier are satisfied.

Kneser showed in \cite{Kneser1957} that it is often possible to compute the class set $\text{Cl}(\Lambda)$ via his method of iterated $p$-neighbours. Let $p$ be prime and $\Lambda$ be an integral lattice in some quadratic space $(V,Q)$ over $\mathbb{Q}$. We say that an integral lattice $\Pi \subset V$ is a $p$-neighbour of $\Lambda$ if $\Lambda/(\Lambda \cap \Pi) \cong \Pi/(\Lambda\cap\Pi) \cong \mathbb{Z}/p\mathbb{Z}$, i.e.\ the intersection has index $p$ in both $\Lambda$ and $\Pi$. We write $\Lambda\sim_p \Pi$ to denote that $\Pi$ is a $p$-neighbour of $\Lambda$. 

It is known that if $\Pi$ is a $p$-neighbour of $\Lambda$ then $\mathrm{disc}(\Lambda) = \mathrm{disc}(\Pi)$ and $\Pi_q = \Lambda_q$ for all primes $q \neq p$. Further, when $p \nmid 2\,\mathrm{disc}(\Lambda)$ then $\Lambda_p \sim \Pi_p$, since both are local lattices that are unimodular with the same discriminant. In this case we then have that $\Pi \in \text{gen}(\Lambda)$, so that $p$-neighbour moves allow us to navigate the genus of $\Lambda$ in a way that can generate new isometry classes. In particular, Kneser showed that if $(V,Q)$ is a non-degenerate quadratic space over $\mathbb{Q}$ of dimension $n\geq 3$, then for any two classes $[\Lambda_1], [\Lambda_2] \in \text{Cl}(\Lambda)$ there exists a finite set of primes $p_1, p_2,\dots, p_s$ and integral lattices $\Pi_i\subset V$ such that:
$$\Lambda \sim_{p_1} \Pi_1 \sim_{p_2} \dots \sim_{p_s} \Pi_s = \Lambda_i.$$ 
The lattices considered in this paper have only one spinor genus, and in this case it is known that we may take $p_i = p$ for all $i$, for any fixed prime $p\nmid 2\,\mathrm{disc}(\Lambda)$. 

As explained in \cite{voight2023kneser}, when $p\nmid 2\,\mathrm{disc}(\Lambda)$ and $n\geq 3$ the $p$-neighbours of $\Lambda$ are in bijection with isotropic lines in $\Lambda/p\Lambda$, and so are easy to compute. In particular, if $[v]\in\Lambda/p\Lambda$ generates an isotropic line then the corresponding $p$-neighbour is: \[\Pi = \mathbb{Z}\left(\frac{v}{p}\right) + \{w\in \Lambda\,|\,T(w,v)\in p\mathbb{Z}\}.\] In particular, the number of $p$-neighbours $N(p)$ is finite. When $n = 3$ we have $N(p) = p+1$ (the number of solutions to the quadric $Q(x) \equiv 0 \bmod p$ in $\mathbb{P}^2(\mathbb{F}_p)$).

Both the genus and the class set of $\Lambda$ have adelic interpretations. Consider the group of adelic points of the global orthogonal group $O(V)$ with respect to the local orthogonal groups $O(\Lambda_p)$, i.e.\: 
$$O(\widehat{V}) = \{(g_p)_p \in \prod_pO(V_p):g_p \in O(\Lambda_p) \text{ for all but finitely many primes } p\}.$$ Then $O(\widehat{V})$ acts transitively on $\text{gen}(\Lambda)$ and the stabiliser of $\Lambda$ is $O(\widehat{\Lambda}) = \prod_p O(\Lambda_p)$. Hence there are bijections: \begin{align*}\text{gen}(\Lambda) &\longleftrightarrow O(\widehat{V})/O(\widehat{\Lambda})\\ \text{Cl}(\Lambda) &\longleftrightarrow O(V)\backslash O(\widehat{V})/O(\widehat{\Lambda}) \end{align*}

As discussed in Hein \cite[Proposition 3.1.12]{hein2016orthogonal}, there is an adelic description of $p$-neighbours. If $p \nmid 2\, \mathrm{disc}(\Lambda)$ and $n\geq 3$ then $V_p$ is isotropic and we can fix a basis for $V_p$ so that $V_p = \mathbb{H} \perp \mathbb{H}^{\perp}$, with $\mathbb{H}\sim\left\langle\begin{pmatrix}0 & 1\\ 1 & 0\end{pmatrix}\right\rangle$ a hyperbolic plane. We can then construct the element $\widehat{p} \in O(\widehat{V})$ such that \[\widehat{p}_{(q)} = \begin{cases}
    \left(\begin{smallmatrix}
    p&0&0&\dots&0\\
    0& \frac{1}{p}&0&\dots&0\\
    0&0&1&\dots&0\\
    \vdots &&\ddots&\ddots&\\
    0&\dots&\dots&0&1
\end{smallmatrix}\right),&\text{if } q = p\\
I_n,&\text{if } q \neq p \end{cases}\] The $p$-neighbours of $\Lambda$ are then shown to be in bijection with coset representatives in the decomposition:
\[
O(\widehat{\Lambda})\, \widehat{p}\, O(\widehat{\Lambda}) = \bigsqcup_{m=1}^{N(p)} \widehat{p}_m O(\widehat{\Lambda}),
\]

The theory of $2$-neighbours is a little trickier to handle due to the fact that the local classification is more involved. However, when $\Lambda$ is even and $v_2(\mathrm{disc}(\Lambda)) \leq 1$ (which will often be the case of interest in this paper), the theory as described above will still work for $p=2$. We refer the reader to \cite{schulze2020lecture} for more information.

\subsection{Ternary Lattices and Quaternion orders}\label{Ternary Lattices and Quaternion orders}

In this subsection we recall the Clifford correspondence between ternary lattices and quaternion orders. We assume that the reader is familiar with the theory of quaternion algebras and their orders. A thorough discussion of these topics is given in \cite{voight2021quaternion}.

Let $\Lambda = \mathbb{Z}e_1\oplus \mathbb{Z}e_2\oplus \mathbb{Z}e_3$ be an even ternary lattice in a quadratic space $(V,Q)$ over $\mathbb{Q}$. Suppose that the Gram matrix of $\Lambda$ in this basis is \[T = \begin{pmatrix} 2a & u & v\\ u & 2b & w\\ v & w & 2c \end{pmatrix}.\] The even Clifford algebra of $V$ is the associative $\mathbb{Q}$-algebra given by: \[B = \mathrm{Cl}^0(V) = \mathbb{Q}\oplus \mathbb{Q}e_1e_2 \oplus \mathbb{Q}e_2e_3\oplus \mathbb{Q} e_1e_3,\] with relations $e_i^2 = Q(e_i)$ and $e_ie_j + e_je_i = T(e_i,e_j)$. It is known that $B$ is a quaternion algebra and is definite if and only if $(V,Q)$ is positive definite. 

The lattice $\Lambda\subset V$ gives rise to an order of discriminant $\mathrm{discrd}(\mathcal{O}) = \frac{1}{2}\mathrm{disc}(\Lambda)$: \[\mathcal{O} = \mathrm{Cl}^0(\Lambda) = \mathbb{Z} \oplus \mathbb{Z}e_1e_2 \oplus \mathbb{Z}e_2e_3 \oplus \mathbb{Z}e_1e_3\subset B.\] Note that it is important that $\Lambda$ is even, since otherwise $\mathrm{Cl}^0(\Lambda)$ may not be closed under multiplication. If $\Lambda$ is not even then we may instead work with a suitable scaling of the quadratic map (although altering the discriminant). 

We wish to give a standard basis for the quaternion algebra $B$. In order to do this we work out the Gram matrix of the trace form $\text{trd}(\alpha\bar{\beta})$ on the basis given above to get the matrix: \[A = \begin{pmatrix}2 & u & w & v\\ u & 2ab & uw-bv & aw\\ w & uw-bv & 2bc & cu\\ v & aw & cu & 2ac\end{pmatrix}.\]

The matrix: \[M = \begin{pmatrix}1 & 0 & 0 & 0\\ -u & 2 & 0 & 0\\ bv-uw & w & u & -2b\\ u^2w-buv-2abw & 2bv-uw & 4ab-u^2 & 0\end{pmatrix}\] is then such that $D = MAM^T$ is the diagonal matrix: \[D = 
\begin{pmatrix}2 & 0 & 0 & 0\\ 0 & 2(4ab-u^2) & 0 & 0\\ 0 & 0 & 2bN & 0\\ 0 & 0 & 0 & 2bN(4ab-u^2)
\end{pmatrix}.\] It is now clear that $B = \left(\frac{u^2-4ab\,,\,-bN}{\mathbb{Q}}\right)$ under the basis $\{1,i,j,k\}$ with: \begin{align*}
    i & = 2e_1e_2-u,\\ j &= we_1e_2 + ue_2e_3 - 2be_1e_3 + (bv-uw),\\ k &= (2bv-uw)e_1e_2 + (4ab-u^2)e_2e_3 + (u^2w-buv-2abw).
\end{align*} (It is easily checked that the standard relations $i^2 = u^2-4ab, j^2 = -bN$ and $ij=k=-ji$ are satisfied).\\
Inverting the change of basis matrix $M$ gives: \[M^{-1} = \begin{pmatrix}1 & 0 & 0 & 0\\ \frac{u}{2} & \frac{1}{2} & 0 & 0\\ \frac{w}{2} & \frac{uw-2bv}{8ab-2u^2} & 0 & \frac{1}{4ab-u^2}\\ \frac{v}{2} & \frac{2aw-uv}{8ab-2u^2} & -\frac{1}{2b} & \frac{u}{8ab^2-2bu^2}
\end{pmatrix}\] so that \begin{align*}e_1e_2 &= \frac{u+i}{2}\\ e_2e_3 &= \frac{w(4ab-u^2) + (uw-2bv)i + 2k}{8ab-2u^2}\\ e_1e_3 &= \frac{bv(4ab-u^2) + b(2aw-uv)i - (4ab-u^2)j + uk}{8ab^2-2bu^2}\end{align*} so that in terms of the standard basis $\{1,i,j,k\}$ we have: \small \[\mathcal{O} = \mathbb{Z}\oplus\mathbb{Z}\frac{u+i}{2}\oplus\mathbb{Z}\frac{w(4ab-u^2) + (uw-2bv)i + 2k}{8ab-2u^2}\oplus \mathbb{Z}\frac{bv(4ab-u^2) + b(2aw-uv)i - (4ab-u^2)j + uk}{8ab^2-2bu^2}.\] 
\normalsize 
It is easily checked that the Gram matrix of this order is equivalent to the one above and that $\mathrm{discrd}(\mathcal{O}) = \frac{1}{2}\mathrm{disc}(\Lambda)$ under this basis, as expected.

The association $\Lambda \rightarrow\mathrm{Cl}^0(\Lambda)$ has been studied many times historically and can be viewed as a generalisation of Gauss' correspondence between integral binary quadratic forms and ideals in quadratic rings. It has many interesting functorial properties, which are explored in Chapter $22$ of \cite{voight2021quaternion}. In particular, the correspondence is functorial in the underlying rings, so that we may compatibly extend the recipe to connect local lattices and local orders. Functoriality implies the existence of a bijection $\text{Cl}(\Lambda) \longleftrightarrow\text{Typ}(\mathcal{O})$ that identifies isometries on the LHS with isomorphisms on the RHS (where $\text{Typ}(\mathcal{O})$ is the set of orders in $B$ that are locally isomorphic to $\mathcal{O}$ everywhere, up to global isomorphism).

\subsection{Classical Modular Forms} \label{subsec: classical MFs}
We assume that the reader is familiar with the theory of modular forms and Hecke operators (see \cite{diamond2005first} and \cite{cohen2017modular} for a detailed introduction). We denote by $S_k(\Gamma_0(N))$ the space of cusp forms of weight $k$ and level $N$, and $T_p$ the usual Hecke operator at $p$ acting on this space. Recall that $S_k(\Gamma_0(N))$ has a basis of eigenforms for the algebra generated by the Hecke operators $T_p$ for $p\nmid N$ (we will not need to consider the other Hecke operators in this paper).

The space of oldforms is given by:
\begin{align*}
S_{k}^{\text{old}}(\Gamma_0(N)) = \sum_{\substack{M \mid N \\M\neq N}} \sum_{d \mid \tfrac{N}{M}} \alpha_d(S_{k}(\Gamma_0(M))).
\end{align*} where $\alpha_d(f)(\tau) = f(d\tau)$ are the degeneracy maps. The new subspace is then defined via the orthogonal direct sum decomposition (with respect to the Petersson inner product): \[S_k(\Gamma_0(N)) = S_k^{\text{old}}(\Gamma_0(N)) \oplus S_k^{\text{new}}(\Gamma_0(N))\] More generally, if $M|N$ then the $M$-new subspace is given by: \[S_k^{M\mathrm{-new}}(\Gamma_0(N)) = \displaystyle\sum_{M\mid M^{\prime}\mid N} \sum_{d \mid \frac{N}{M^{\prime}}} \alpha_d(S_k^{\text{new}}(\Gamma_0(M^{\prime})).\]

We also recall that there are additional involutions acting on the spaces $S_k(\Gamma_0(N))$. For each $d$ where $\gcd(d,N/d) = 1$, the Atkin-Lehner involution $W_d$ acts on $S_{k}(\Gamma_0(N))$ via $f \mapsto f|_{k}[W_d]$ where:    	
\begin{align*}
    W_d = \left(\begin{smallmatrix}
        dx & y \\
        Nz & d
    \end{smallmatrix}\right) \, \text{ for integers $x,y,z$ such that } dx - \tfrac{Nyz}{d} = 1.
\end{align*}
Here the weight $k$ slash operator is given by $(f|_{k} [\gamma])(\tau) = \text{det}(\gamma)^{k-1}(c\tau + d)^{-k} f(\gamma \cdot \tau)$ with $\gamma\cdot \tau$ being the Mobius action of $\gamma = \begin{pmatrix}a & b\\ c & d\end{pmatrix}\in\text{GL}_2^+(\mathbb{Q})$ on $\tau\in\mathcal{H}$ (complex upper half plane). 

In order to keep track of Atkin-Lehner eigenvalues, the following notation from \cite{heintornariavoight} will be useful. Let $N$ and $d$ be positive integers. Then the corresponding sign vector is $\varepsilon^{(N)}_d= (\eta_p)_{p^e\mid \mid N}$ where \[\eta_p = \begin{cases}+1 & \text{ if } p\nmid d\\ -1 & \text{ if } p\mid d\end{cases}.\]
Sign vectors will later be multiplied component-wise.

For each sign vector $\varepsilon^{(N)}_d$, we have eigenspaces of the form
\begin{align*}\label{eq:neweps}
  S_{k}\!\left(\Gamma_0(N), \varepsilon^{(N)}_d\right)
  \;:=\;
  \Bigl\{
    f \in S_{k}\!\left(\Gamma_0(N)\right)
    :
    W_q f = \eta_p f
    \text{ for all } q = p^e \| N
  \Bigr\}.
\end{align*} Since the Atkin-Lehner involutions $W_d$ commute with degeneracy operators and the Hecke operators $T_p$ when $p \nmid N$ and $\gcd(p,d) = 1$, this induces decompositions of Hecke-modules for each $M|N$:  \[
  S^{M\mathrm{-new}}_{k}\!\left(\Gamma_0(N)\right)
  \;=\;
  \bigoplus_{\varepsilon^{(N)}_d}
  S^{M\mathrm{-new}}_{k}\!\left(\Gamma_0(N), \varepsilon^{(N)}_d\right).
\] where $\varepsilon_d^{(N)}$ runs over all possible sign vectors corresponding to divisors $d \,|\, N$.

\subsection{Orthogonal Modular Forms}\label{subsec: OMFs}

We now introduce the theory of (definite) orthogonal modular forms. This is a special case of the wider theory of algebraic modular forms introduced by Gross in \cite{gross1999algebraic} and discussed in detail for our particular case in \cite{greenbergvoight}. Let $(V,Q)$ be a positive definite quadratic space over $\mathbb{Q}$ and $\Lambda\subset V$ be a lattice. Also let $W$ be a finite-dimensional representation of $O(V)$. The corresponding space of orthogonal modular forms of weight $W$ and level $O(\widehat{\Lambda})$ is then defined to be:

\[\mathcal{M}(O(V),O(\widehat{\Lambda}),W) = \left\{\varphi : O(\widehat{V}) \to W \,\Bigg|\, \begin{array}{c}\varphi(\gamma gk) = \gamma \varphi(g) \text{ for all }\\ (\gamma, g, k) \in O(V)\times O(\widehat{V})\times O(\widehat{\Lambda})\end{array}\right\}.\] These spaces are finite-dimensional, since any such function is completely determined by its values on $O(V)\backslash O(\widehat{V})/O(\widehat{\Lambda}) \leftrightarrow \text{Cl}(\Lambda)$.

For each $\hat{g}\in O(\widehat{V})$ there exists a Hecke operator $T_{\hat{g}}$ acting on $\mathcal{M}(O(V),O(\widehat{\Lambda}),W)$. In direct analogy with classical Hecke operators, one considers the (finite) double coset decomposition: \[O(\widehat{\Lambda})\,\hat{g}\,O(\widehat{\Lambda}) = \coprod_{i=1}^{r(\hat{g})}\hat{g}_i O(\widehat{\Lambda})\] and defines: \[T_{\hat{g}}(\varphi)(g) = \sum_{i=1}^{r(\hat{g})}\varphi(gg_i).\]

Suppose from now on that $n \geq 3$. If we take $\hat{g} = \hat{p}$ for any prime $p\nmid 2\,\mathrm{disc}(\Lambda)$ (where $\hat{p}\in O(\widehat{V})$ is as defined previously) then we have a family of Hecke operators $T_{\hat{p}}$ acting on $\mathcal{M}(O(V),O(\widehat{\Lambda}),W)$. We will often simply write these operators as $T_p$ for reasons that will become clear soon. The algebra generated by these operators is commutative. 

Recall that the coset representatives for $T_p$ can be identified with $p$-neighbours of $\Lambda$. Under the identification $O(\widehat{V})/O(\widehat{\Lambda}) \leftrightarrow \text{gen}(\Lambda)$ we then find that this operator simply averages the value of $\phi$ over the $p$-neighbours of the input lattice. In the special case where $W = \mathbb{C}$ is the trivial representation then \[\mathcal{M}(O(V),O(\widehat{\Lambda}),W) \cong \{\varphi:\text{Cl}(\Lambda) \rightarrow\mathbb{C}\} \cong \mathbb{C}^h\] and under the standard basis the operator $T_p$ simply counts the number of $p$-neighbours of $\Lambda_i$ that are isometric to $\Lambda_j$ (where $\text{Cl}(\Lambda) = \{[\Lambda_1],[\Lambda_2],...,[\Lambda_h]\}$). Alternatively, these operators represent the adjacency matrices of the (directed) $p$-neighbour graphs (with vertices for each of the classes $[\Lambda_i]$ and directed edges from $[\Lambda_i]$ to $[\Lambda_j]$ for each $p$-neighbour of $[\Lambda_i]$ that is isometric to $[\Lambda_j]$). 

It is known that the space $\mathcal{M}(O(V),O(\widehat{\Lambda}),W)$ has a basis of simultaneous eigenforms for the Hecke algebra generated by the operators $T_p$ for $p\nmid 2\,\mathrm{disc}(\Lambda)$ prime. Again, a lot of what is described above still works if we include the prime $p=2$ (e.g.\ if $\Lambda$ is even and $v_2(\mathrm{disc}(\Lambda)) \leq 1$).
    
Note that the constant function $\mathbbm{1}$ is an eigenvector for every Hecke operator $T_{\hat{g}}$, having eigenvalue $r(\hat{g})$ (i.e.\ when $\hat{g} = \hat{p}$ for $p\nmid 2\,\mathrm{disc}(\Lambda)$ the eigenvalue is $N(p)$, the total number of $p$-neighbours of $\Lambda$).

As with spaces of classical modular forms, it is possible to decompose into the direct product of an ``Eisenstein subspace" and a ``cuspidal subspace" (even though the locally symmetric space in this case has no cusps!). Without giving the full definition we remark that the Eisenstein subspace can only be non-trivial when the weight $W = \mathbb{C}$ is the trivial representation (see \cite[\S1.2]{rama2020computation}), and in all cases considered in this paper it is simply the one-dimensional space spanned by $\mathbbm{1}$. The cuspidal subspace $\mathcal{S}(O(V), O(\hat{\Lambda}), W)$ is then defined to be the full space $\mathcal{M}(O(V),O(\hat{\Lambda}), W)$ if $W$ is a non-trivial weight and the span of the Hecke eigenforms that are not $\mathbbm{1}$ when $W = \mathbb{C}$ is trivial (alternatively, the orthogonal complement of $\mathbbm{1}$ with respect to a suitable inner product, defined in \cite{hein2016orthogonal}). The Hecke operators are known to respect this decomposition.
    
Later we will make heavy use of theorems of Hein-Tornaria-Voight \cite{heintornariavoight}, proving the existence of Hecke-equivariant isomorphisms between certain spaces of orthogonal modular forms and spaces of classical modular forms. We will introduce these now.

From now on we specialise to the case of $n = 3$ (the case of interest in this paper). In this case it is well known that: \[O(V) \twoheadrightarrow O(V)/\{\pm I\} \cong SO(V) \cong B^{\times}/\mathbb{Q}^{\times}.\] Here $B = \mathrm{Cl}^0(V)$ is the even Clifford algebra of $V$, which in our case is a definite quaternion algebra over $\mathbb{Q}$. Any representation of $B^{\times}/\mathbb{Q}^{\times}$ can be lifted to one for $O(V)$ via the surjection above. 

One such family of representations comes from choosing a splitting: \[B^{\times}/\mathbb{Q}^{\times}\hookrightarrow \text{PGL}_2(\mathbb{C}).\] For each integer $k\geq 0$ the irreducible representation $W_{2k} = \text{Sym}^{2k}(\mathbb{C})\otimes\text{det}^{-k}$ of $\text{GL}_2(\mathbb{C})$ descends to $\text{PGL}_2(\mathbb{C})$ and so gives rise to an irreducible representation of $O(V)$. This representation can also be constructed by considering the action of $O(V)$ on a suitable space of harmonic polynomials on $V$ with respect to an explicit Laplacian $\Delta_Q$. See Example $3.5.3$ of \cite{hein2016orthogonal} for more details.

Secondly, we can construct certain characters of $O(V)$. Applying the reduced norm map gives \[\text{nrd}:B^{\times}/\mathbb{Q}^{\times} \longrightarrow \mathbb{Q}^{\times}/(\mathbb{Q}^{\times})^2.\] For each $d\in\mathbb{N}$ we get a character: \begin{align*} \phi_d:\mathbb{Q}^{\times}/(\mathbb{Q}^{\times})^2 &\longrightarrow \{\pm 1\}\\ a &\longmapsto \prod_{p\mid d}(-1)^{v_p(a)}. \end{align*} The composite $\theta_d = \phi_d \circ \text{nrd}$ defines spinor characters of $O(V)$ for each choice of $d$.

The following results will come in handy later (with the spaces $M(SO(V),SO(\widehat{\Lambda}),W)$ being defined similarly to above).

\begin{lem}\label{OtoSO}
For each lattice $\Lambda\subset V$ there is a Hecke-equivariant isomorphism: \[\mathcal{M}(O(V), O(\widehat{\Lambda}),W_{2k}\otimes \theta_d) \cong M(SO(V),SO(\widehat{\Lambda}), W_{2k}\otimes \theta_d)\]
\end{lem}

\begin{proof}
The restriction map from $O(\widehat{V})$ to $SO(\widehat{V})$ induces a Hecke-equivariant injective homomorphism \[\mathcal{M}(O(V),O(\widehat{\Lambda}),W) \longrightarrow M(SO(V),SO(\widehat{\Lambda}),W)\] for any representation $W$ of $O(V)$. The image consists of functions \[\varphi\in M(SO(V),SO(\widehat{\Lambda}),W)\] such that $\varphi(g) = \varphi((-I)\cdot g) = (-I)\cdot \varphi(g)$ for every $g\in SO(\widehat{V})$. When $-I$ acts trivially on $W$ this condition is automatically satisfied and so we get a Hecke-equivariant isomorphism. All of the representations constructed above have this property, establishing the result. 
\end{proof}

Before moving on to the main result of Hein-Tornaria-Voight we need to recall one more definition. Let $\mathcal{O}$ be an order in a quaternion algebra $B$ over $\mathbb{Q}$ and $p \mid \mathrm{discrd}(\mathcal{O})$ be prime. Then the Eichler invariant $\varepsilon_p(\mathcal{O})$ takes values in $\{-1,0,+1\}$ and is defined by (where $\text{rad}(\mathcal{O}_p)$ is the Jacobson radical of the local order $\mathcal{O}_p = \mathcal{O}\otimes_{\mathbb{Z}}\mathbb{Z}_p$):
$$\varepsilon_p(\mathcal{O}) = \begin{cases}
 +1, \quad &\text{if }\mathcal{O}_p/\text{rad}(\mathcal{O}_p) \cong \mathbb{F}_p \times \mathbb{F}_p \text{ \quad ($\mathcal{O}$ is \textit{residually split} at } p),\\
 0, \quad &\text{if }\mathcal{O}_p/\text{rad}(\mathcal{O}_p) \cong \mathbb{F}_p \quad (\text{$\mathcal{O}$ is \textit{residually ramified} at $p$),}\\
 -1, \quad &\text{if }\mathcal{O}_p/\text{rad}(\mathcal{O}_p) \cong \mathbb{F}_{p^2} \quad (\text{ $\mathcal{O}$ is \textit{residually inert} at }p).
\end{cases}$$
 The orders featuring in the main results of \cite{heintornariavoight} are residually unramified everywhere, i.e.\ having local Eichler invariant $\varepsilon_p(\mathcal{O}) \in \{\pm 1\}$ for all primes $p \mid \mathrm{discrd}(\mathcal{O})$. In particular, Eichler orders are examples of such orders, since if $\mathrm{discrd}(\mathcal{O}) = DM$ with $D = \mathrm{disc}(B)$ and $M$ coprime to $D$ then it can be shown that \[\varepsilon_p(\mathcal{O}) = \begin{cases}
    -1, &\text{if } p \mid D, \\
     1, &\text{if } p \mid M
\end{cases}.\] A complete classification of local orders at odd primes via their Eichler invariants is possible, and is given in Section $24.6$ of \cite{voight2021quaternion}. The analogous classification for $p=2$ is given in \cite{lemurell}.

The following main result is then a consequence of the above lemma, along with Theorem 4.10, Theorem 4.12 and Corollary 6.12 in \cite{heintornariavoight}.

\begin{thm}\label{thm: 6.12 in HTV}
Suppose that $\Lambda\subset V$ is an even lattice in a positive definite ternary quadratic space $(V,Q)$ over $\mathbb{Q}$. Suppose that the associated quaternion algebra $B = \textup{Cl}^0(V)$ has discriminant $D$ and that the order $\mathcal{O} = \textup{Cl}^0(\Lambda)$ is residually unramified everywhere of discriminant $N = \textup{discrd}(\mathcal{O}) = \frac{1}{2}\textup{disc}(\Lambda)$. Let $N^+$ and $N^-$ be the products of the prime powers $p^e||N$ that satisfy $\varepsilon_p(\mathcal{O}) = 1$ and $\varepsilon_p(\mathcal{O}) = -1$ respectively. 

Then for each $k \geq 0$ and each square-free $d\mid N$ we have a Hecke-equivariant isomorphism (for all Hecke operators $T_p$ with $p \nmid N$):
$$\mathcal{S}(O(V), O(\widehat{\Lambda}), W_{2k}\otimes \theta_d) \cong 
\bigoplus_{M|N^-} S_{2k+2}^{M-\textup{new}}(\Gamma_0(MN^{+}), \varepsilon_D^{(MN^{+})}\varepsilon_d^{(MN^{+})})$$
where the sum runs over $M|N^-$ with $v_p(M) \equiv v_p(N) \bmod 2$ for all primes $p \,|\, N^-$.
\end{thm}

Note that $D = \mathrm{disc}(B)$ is known to be square-free and to divide $N^-$ and so it makes sense to construct the sign vector $\varepsilon_D^{(N)}$.

The aim of this paper is to generalise this result to the cases of even $\Lambda$ of half-discriminant $N^2$ with $N$ odd and square-free. Such lattices cannot be handled using the above results because the underlying orders are residually ramified at least one prime. Our initial strategy for proving Hecke-equivariant isomorphisms between the spaces $\mathcal{S}(O(V),O(\widehat{\Lambda}),W_{2k}\otimes\theta_d)$ and spaces of classical modular forms will be to shift focus to various sublattices and superlattices $\Pi\subset V$ satisfying $O(\widehat{\Lambda}) = O(\widehat{\Pi})$. On the level of Hecke modules, doing this makes no change.

\begin{lem}\label{samespace}
Suppose $\Lambda,\Pi\subset V$ are lattices such that $O(\widehat{\Lambda}) = O(\widehat{\Pi})$. Then for each representation $W$ of $O(V)$ there is an equality of Hecke modules: \[\mathcal{M}(O(V),O(\widehat{\Lambda}),W) = \mathcal{M}(O(V),O(\widehat{\Pi}),W).\]
\end{lem}

\begin{proof}
Since $O(\widehat{\Lambda}) = O(\widehat{\Pi})$ is follows that $\mathcal{M}(O(V),O(\widehat{\Lambda}),W) = \mathcal{M}(O(V),O(\widehat{\Pi}),W)$ as spaces. We also have that: \[O(\widehat{\Lambda})\,g\, O(\widehat{\Lambda}) = O(\widehat{\Pi})\,g\, O(\widehat{\Pi})\] for any $g\in O(\widehat{V})$. This implies that we get the same Hecke representatives and hence the same Hecke operators on both sides.
\end{proof}

We will often find ourselves in a position where the related lattice $\Pi$ can be viewed as a lattice of smaller discriminant upon scaling of the underlying quadratic form. Once again, this provides an equivalent Hecke module.

\begin{lem}\label{scaling}
Let $\Lambda\subset V$ be a lattice in a quadratic space $(V,Q)$ over $\mathbb{Q}$. The Hecke module $\mathcal{M}(O(V),O(\widehat{\Lambda}),W)$ is invariant under non-zero scaling of the underlying quadratic map $Q$.
\end{lem}

\begin{proof}
This is also clear. For any $\alpha\in\mathbb{Q}^{\times}$ the scaled quadratic space $(V,Q^{(\alpha)})$ has the same orthogonal group as $(V,Q)$. The same is true for $(\widehat{V},Q^{(\alpha)})$ and $(\widehat{V},Q)$. Consequently, the subgroup $O(\widehat{\Lambda})$ is invariant under scaling of the underlying quadratic map, which by Lemma \ref{samespace} implies an equality of spaces of orthogonal modular forms and Hecke operators.
\end{proof}

\section{Some local lattices, linear maps and quaternion orders}

As mentioned above, we will be interested in even ternary lattices of discriminant $2N^2$ with $N$ odd and square-free. It will be handy to classify the corresponding local lattices of such discriminants (which must be integral by the global even assumption).

\begin{lem}\label{localisom}
    
Let $\Lambda_p$ be a ternary integral $\mathbb{Z}_p$-lattice of discriminant $\mathrm{disc}(\Lambda_p) = [2N^2]$ in a quadratic space $(V_p,Q)$ over $\mathbb{Q}_p$. If $p = 2$ then assume that $\Lambda_p$ is also even. Then we must have:

\[\Lambda_p\sim \begin{cases} \langle 1,1,2\rangle & \text{ if } p\nmid 2N\\ \left\langle \begin{pmatrix}0 & 1\\ 1 & 0\end{pmatrix},-2\right\rangle & \text{ if } p = 2\\  \langle 1,p,p\rangle,  \langle r,p,pr\rangle,  \langle 1,1,p^2\rangle,  \langle 1,r,p^2r\rangle & \text{ if } p | N \text{ and } p\equiv 1,7\bmod 8\\  \langle r,p,p\rangle,  \langle 1,p,pr\rangle,  \langle 1,r,p^2\rangle,  \langle 1,1,p^2r\rangle & \text{ if } p \,|\, N \text{ and } p\equiv 3,5\bmod 8\end{cases}.\] Further, we have that $S(V_p) = 1$ unless $p=2$ or: \[\Lambda_p \sim \begin{cases} \langle r,p,pr\rangle  \text{ and } p = 1 \bmod 8\\ \langle r,p,p\rangle  \text{ and } p = 3 \bmod 8\\ \langle 1,p,pr\rangle  \text{ and } p = 5 \bmod 8\\ \langle 1,p,p\rangle  \text{ and } p = 7 \bmod 8\end{cases}\]
\end{lem}

\begin{proof}
If $p\nmid 2N$ then $\Lambda_p$ is unimodular of discriminant $\mathrm{disc}(\Lambda_p) = [2N^2] = [2]$ and so $\Lambda_p\sim \langle 1,1,2\rangle$.

If $p = 2$ then since $v_2(2N^2) = 1$ we must have a unimodular Jordan component of rank $2$ and a $2$-modular Jordan component of rank $1$. Since $\Lambda_p$ is even and $\mathrm{disc}(\Lambda_p) = [2N^2] = [2]
$ (as $N$ is odd), the $2$-adic classification of lattices tells us that \[\Lambda_p \sim \left\langle\begin{pmatrix}0 & 1\\ 1 & 0\end{pmatrix}, -2\right\rangle \qquad \text{ or } \qquad \left\langle\begin{pmatrix}2 & 1\\ 1 & 2\end{pmatrix}, 6\right\rangle.\] However, these lattices are isometric, since the Conway-Sloane $2$-adic symbols are $1_{\text{II}}^{+2}2_7^{+1}$ and $1_{\text{II}}^{-2}2_3^{-1}$, which are related by the Sign Walking rule.

If $p \,|\, N$ then $v_p(2N^2) = 2$, since $N$ is odd and square-free. Hence $\Lambda_p$ either has a unimodular Jordan component of rank $1$ and a $p$-modular Jordan component of rank $2$, or a unimodular Jordan component of rank $2$ and a $p^2$-modular Jordan component of rank $1$, i.e.\ \[\Lambda_p\sim \langle a,p,pb\rangle \qquad \text{ or } \qquad \Lambda_p \sim \langle 1,a,p^2b\rangle\] for some $a,b\in\{1,r\}$. Since $\mathrm{disc}(\Lambda_p) = [2N^2] = [2p^2]$ we find that $[ab] = [2]$, so that $[a]=[b]$ when $p \equiv 1,7\bmod 8$ and $[a]=[rb]$ when $p \equiv 3,5\bmod 8$. This reduces the list to the four possibilities given above in each case.
\end{proof}

Suppose that $V_p$ is a three dimensional $\mathbb{Q}_p$-vector space and fix a basis $\{v_1,v_2,v_3\}$. We will later rely heavily on the following simple linear map (on coordinates under this basis): \[\psi_{(p)}(x_1,x_2,x_3) = (px_1,x_2,x_3).\]

\begin{prop}\label{psip}
Let $\Lambda_p$ be a ternary $\mathbb{Z}_p$-lattice in a quadratic space $(V_p,Q)$ over $\mathbb{Q}_p$ such that $\Lambda_p\sim\langle a,p,pb\rangle$ for some $a,b\in\{1,r\}$. Then $\Pi_p = \psi_{(p)}(\Lambda_p)$ is an index $p$ sublattice such that $O(\Lambda_p) = O(\Pi_p)$.
\end{prop}

\begin{proof}
It is clear that $\Pi_p$ is an index $p$ sublattice of $\Lambda_p$ since $\text{det}(\psi_{(p)}) = p$. Note that since $\Lambda_p$ is an elementary lattice (i.e.\ $\Lambda_p^*/\Lambda_p$ is an elementary $p$-group), Theorem $1.1(1)$ of Allcock \cite{allcock2025} tells us that $O(\Lambda_p)$ is maximal in $O(V_p)$ among the orthogonal groups of all $p$-adic lattices. Hence, if we can show that $O(\Lambda_p)\subseteq O(\Pi_p)$ then we will be done.

By standard arguments, there are exactly $p^2+p+1$ sublattices of $\Lambda_p$ of index $p$ up to change of basis. If $\Lambda_p$ has basis $\{v_1,v_2,v_3\}$ with respect to the above Gram matrix then these sublattices have bases \[\{pv_1,v_2,v_3\} \qquad \{v_1+\alpha v_2, pv_2, v_3\} \qquad \{v_1+\alpha v_3, v_2 + \beta v_3, pv_3\},\] with $\alpha,\beta\in \{0,1,...,p-1\}$, and corresponding Gram matrices: \[\begin{pmatrix}
p^2a & 0 & 0\\ 0 & p & 0\\ 0 & 0 & pb
\end{pmatrix}\qquad \begin{pmatrix}
a+p\alpha^2 & p^2\alpha & 0\\ p^2\alpha & p^3 & 0\\ 0 & 0 & pb
\end{pmatrix} \qquad \begin{pmatrix}
a+pb\alpha^2 & pb\alpha\beta & p^2b
\alpha\\ pb\alpha\beta & p(1+b\beta^2) & p^2b\beta\\ p^2b\alpha & p^2b\beta & p^3b
\end{pmatrix}.\] Note that the first Gram matrix corresponds to the lattice $\Pi_p$, and is visibly the only one having norm ideal $p\mathbb{Z}_p$. Now if $\gamma\in O(\Lambda_p)$ then $\gamma\Lambda_p = \Lambda_p$ and $\gamma\Pi_p$ is a sublattice of $\Lambda_p$ of index $p$. However, isometries preserve norm ideals and so we must have that $\gamma\Pi_p = \Pi_p$, i.e.\ $\gamma\in O(\Pi_p)$. This gives $O(\Lambda_p)\subseteq O(\Pi_p)$, as required.
\end{proof}

Suppose that $(V_q)_q$ is a collection of three dimensional $\mathbb{Q}_q$-vector spaces, one for each prime $q$. For a fixed prime $p$ and basis $\{v_1,v_2,v_3\}$ for $V_p$ we define the map $\hat{\psi}_{(p)}: \prod_p V_p \rightarrow \prod_p V_p$ which applies the map $\psi_{(p)}$ in component $p$ and the identity map in all other components. 

We will later wish to apply the map $\hat{\psi}_{(p)}$ to a $\mathbb{Z}$-lattice $\Lambda \subset V$ with local lattices $\Lambda_q$ having specified Gram matrices. We define $\hat{\psi}_{(p)}(\Lambda) =\hat{\psi}_{(p)}(\prod_q \Lambda_q)\subset \prod_q V_q$ (under local bases corresponding to the given Gram matrices).

Finally, we will need the following result about local quaternion algebras and orders coming from certain local ternary lattices.

\begin{lem}\label{ram}
Let $\Lambda_p$ be a ternary lattice in a quadratic space $(V_p,Q)$ over $\mathbb{Q}_p$, let $B_p = \mathrm{Cl}^0(V_p)$ be the associated quaternion algebra and $\mathcal{O}_p = \mathrm{Cl}^0(\Lambda_p)$ be the associated local order.

\vspace{0.1in}
\begin{itemize}
\item{If $p$ is odd and $\Lambda_p \sim \langle 1,1,a \rangle$ for some $a\in\{1,r\}$, then $B_p\cong M_2(\mathbb{Q}_p)$ is unramified and $\mathcal{O}_p\cong M_2(\mathbb{Z}_p)$ is maximal of discriminant $1$ (hence unramified).}
\vspace{0.1in}
\newpage
\item{If $p$ is odd and $\Lambda_p \sim \langle 1,a,pb \rangle$ for some $a,b\in\{1,r\}$, then $B_p$ is ramified if and only if $S(V_p) = -1$ and $p\equiv 1 \bmod 4$ or $S(V_p) = 1$ and $p \equiv 3 \bmod 4$. In the ramified case, the order $\mathcal{O}_p$ is the unique maximal order of discriminant $p$ (unramified with $\varepsilon_p(\mathcal{O}) = -1$), otherwise $\mathcal{O}_p\cong\begin{pmatrix}\mathbb{Z}_p & \mathbb{Z}_p\\ p\mathbb{Z}_p & \mathbb{Z}_p\end{pmatrix}$ is an Eichler order of discriminant $p$ (unramified with  $\varepsilon_p(\mathcal{O}) = 1$).}
\vspace{0.1in}
\item{If $p$ is odd and $\Lambda_p \sim \langle 1,a,p^2b\rangle$ for some $a,b\in\{1,r\}$, then $B_p$ is unramified. We have that $\mathcal{O}_p\cong\begin{pmatrix}\mathbb{Z}_p & \mathbb{Z}_p\\ p^2\mathbb{Z}_p & \mathbb{Z}_p\end{pmatrix}$ is an Eichler order of discriminant $p^2$ when $p\equiv 1 \bmod 4$ and $a = 1$ or $p\equiv 3\bmod 4$ and $a= r$ (hence unramified with $\varepsilon_p(\mathcal{O}) = 1$) and is a non-Eichler order of discriminant $p^2$ otherwise (unramified with $\varepsilon_p(\mathcal{O}) = -1$).}
\vspace{0.1in}
\item{If $p=2$ and $\Lambda_p\sim\left\langle\begin{pmatrix}0 & 1\\ 1 & 0\end{pmatrix},2a\right\rangle$ for some $a\in\{1,3,5,7\}$, then $B_p\cong M_2(\mathbb{Q}_2)$ is unramified and $\mathcal{O}_p\cong M_2(\mathbb{Z}_2)$ is maximal (unramified with $\varepsilon_p(\mathcal{O}) = 1$).}
\end{itemize}
\end{lem}

\begin{proof}
If $p$ is odd and $\Lambda_p\sim \langle 1,1,a\rangle$ then by the Clifford correspondence we would have $\mathrm{discrd}(\mathcal{O}_p) = 1$, so that $B_p \cong M_2(\mathbb{Q}_p)$ is unramified and $\mathcal{O}_p\cong M_2(\mathbb{Z}_p)$ is maximal (since for $B_p$ ramified the unique maximal order has reduced discriminant $p$). This order has Eichler invariant $1$ by the classification of orders in local quaternion algebras.

If $p$ is odd and $\Lambda_p\sim \langle 1,a,pb\rangle$ then $B_p = \left(\frac{-a/4,-a(pab)/4}{\mathbb{Q}_p}\right)\cong \left(\frac{-a,-pb}{\mathbb{Q_p}}\right)$. Note that $(-a,-pb)_p = (-1,-1)_p(-1,pb)_p(a,-1)_p(a,pb)_p = (-1,pb)_pS(V_p)$. Since $B_p$ is ramified if and only if $(-a,-pb)_p = -1$, we either require $S(V_p) = -1$ and $(-1,pb) = 1$ (i.e.\ $p\equiv 1 \bmod 4$) or $S(V_p) = 1$ and $(-1,pb)_p = -1$ (i.e.\ $p\equiv 3 \bmod 4$). Note that by the Clifford correspondence we have $\mathrm{discrd}(\mathcal{O}_p) = p$, and by the classification of orders in local quaternion algebras, we must have that $\mathcal{O}_p$ is maximal with Eichler invariant $-1$ when $B_p$ is ramified and Eichler with Eichler invariant $1$ when $B_p$ is unramified.

If $p$ is odd and $\Lambda_p\sim\langle 1,a,p^2b\rangle$ then $B_p = \left(\frac{-a/4,-a(p^2ab)/4}{\mathbb{Q}_p}\right)$. Note that \[(-a/4,-p^2a^2b/4)_p = (-a,-b)_p = 1\] and so $B_p\cong M_2(\mathbb{Q}_p)$ is unramified. Note that under the Clifford correspondence we must have $\mathrm{discrd}(\mathcal{O}_p) = p^2$. By the classification of orders in local quaternion algebras, there are three possibilities up to isomorphism, two unramified (with Eichler invariants $\pm 1$) and one ramified (with $\varepsilon_p(\mathcal{O}_p) = 0$). Exercise $10$ of Chapter $24$ of Voight \cite{voight2021quaternion} tells us that the Eichler invariant in this case is given by the Legendre symbol $\left(\frac{-a}{p}\right)$, which leads to the explicit conditions given in the result.

If $p=2$ and $\Lambda_p \sim \left\langle\begin{pmatrix}0 & 1\\ 1 & 0\end{pmatrix},2a\right\rangle$ then by the Clifford correspondence we have $\mathrm{discrd}(\mathcal{O}_p) = 1$ and so $B_p\sim M_2(\mathbb{Q}_2)$ is unramified and $\mathcal{O}_p \cong M_2(\mathbb{Z}_2)$ is maximal (again since for $B_p$ ramified the unique maximal order has reduced discriminant $p$). This order has $\varepsilon_p(\mathcal{O}_p) = 1$ by the classification of orders in local quaternion algebras.
\end{proof}

\section{Transfer for Hecke modules of orthogonal modular forms.}

From now on let $(V,Q)$ be a positive definite, ternary quadratic space over $\mathbb{Q}$ and $\Lambda \subset V$ be an even lattice of discriminant $2N^2$ with $N$ odd and square-free. We are interested in proving explicit transfer results for the spaces of orthogonal modular forms $\mathcal{S}(O(V),O(\widehat{\Lambda}),W_{2k} \otimes \theta_d)$ with $k \geq 0$ and $d | N$.

\subsection{$N = p$ prime}

We warm up with the case of $\mathrm{disc}(\Lambda) = 2p^2$ (with $p$ odd by assumption). This case has already been fully solved in the recent paper of De Lauri-Fretwell \cite{delaurifretwell}, but we will recap the argument here since it will prepare us for more general cases.

\begin{lem}
$\Lambda$ belongs to a unique genus with local lattices:

\[\Lambda_t \sim \begin{cases}\langle 1,1,2\rangle & \text{ if } t\nmid 2p\\ \left\langle\begin{pmatrix}0 & 1\\ 1 & 0\end{pmatrix},-2\right\rangle & \text{ if } t = 2\\ \langle r,p,pr\rangle & \text{ if } t = p \equiv 1 \bmod 8\\ \langle r,p,p\rangle & \text{ if } t = p\equiv 3 \bmod 8\\ \langle 1,p,pr\rangle & \text{ if } t = p\equiv 5 \bmod 8\\ \langle 1,p,p\rangle & \text{ if } t = p\equiv 7 \bmod 8\end{cases}\]
\end{lem}

\begin{proof}

The first two isometries follow from Lemma \ref{localisom}. Note that $S(V_{\infty}) = 1$ since $(V,Q)$ is positive definite. Since $S(V_t) = 1$ for all $t \nmid 2p$ and $S(V_2) = -1$ we must have $S(V_p) = -1$ (by compatibility of Hasse invariants). Looking at Lemma \ref{localisom} we see that $\Lambda_p$ is uniquely determined up to isometry.
\end{proof}

We can now prove the following transfer result for Hecke modules of orthogonal modular forms attached to this genus.

\begin{thm}\label{2p2}
Let $\Lambda$ be as above. Then for each $k\geq 0$ and $d \,|\, p$ there is a Hecke-equivariant isomorphism (for Hecke operators away from $p$): \[\mathcal{S}(O(V),O(\widehat{\Lambda}),W_{2k}\otimes\theta_d) \cong S_{2k+2}^{\textup{new}}(\Gamma_0(p),\varepsilon_p^{(p)}\varepsilon_d^{(p)})\]
\end{thm}

\begin{proof}
Since $\Lambda_p \sim \langle a,p,pb\rangle$ for some $a,b\in\{1,r\}$ we can apply $\hat{\psi}_{(p)}$ to $\Lambda$ to get a new even lattice $\Pi = V \cap \hat{\psi}_{(p)}(\Lambda)$ with $\mathrm{disc}(\Pi) = 2p^4$ and satisfying $O(\widehat{\Lambda}) = O(\widehat{\Pi})$ (since $O(\Lambda_p) = O(\Pi_p)$ by Proposition \ref{psip} and $\Lambda_q = \Pi_q$ for all $q \neq p$).

Scaling the quadratic map allows us to view $\Pi$ as an even lattice in $(V,Q^{(\frac{1}{p})})$ with $\mathrm{disc}(\Pi) = 2p$ and local lattices:

\[\Pi_t \sim\begin{cases} \langle p^{-1}, p^{-1}, 2p^{-1}\rangle \sim \langle 1,1,2p\rangle & \text{ if } t \nmid 2p\\ \left\langle \begin{pmatrix}0 & p^{-1}\\ p^{-1} & 0\end{pmatrix}, -2p^{-1}\right\rangle \sim  \left\langle\begin{pmatrix}0 & 1\\ 1 & 0\end{pmatrix}, -2p\right\rangle & \text{ if } t = 2\\ \langle 1,r,pr\rangle & \text{ if } t = p\equiv 1 \bmod 8\\ \langle 1,1,pr\rangle & \text{ if } t = p\equiv 3 \bmod 8\\ \langle 1,r,p\rangle & \text{ if } t=p\equiv 5 \bmod 8\\ \langle 1,1,p\rangle & \text{ if } t=p\equiv 7 \bmod 8\end{cases}\]

By Lemma \ref{samespace} and Lemma \ref{scaling} there is an equality of Hecke modules: \[\mathcal{S}(O(V),O(\widehat{\Lambda}), W) = \mathcal{S}(O(V),O(\widehat{\Pi}), W)\] for any representation $W$ of $O(V)$.

By Lemma \ref{ram} the quadratic space $(V,Q^{(\frac{1}{p})})$ gives rise to a quaternion algebra $B = \mathrm{Cl}^0(V)$ of discriminant $p$, and the order $\mathcal{O} = \mathrm{Cl}^0(\Pi)$ is maximal with $\mathrm{discrd}(\mathcal{O}) = p$, hence residually unramified everywhere with $\varepsilon_p(\mathcal{O}) = -1$ (also by Lemma \ref{ram}). The result now follows from the results of Hein-Tornaria-Voight referenced earlier (see Theorem \ref{thm: 6.12 in HTV}).
\end{proof}

\begin{example}
We give an example of Theorem \ref{2p2} with trivial weight, using the Algebraic Modular Forms package in Magma \cite{magma}. We refer to \cite{hein2016orthogonal} for detailed discussions of how these calculations are done in practice.

Consider the even lattice $\Lambda$ with Gram matrix:

\[T = \begin{pmatrix}4 & 1 & -1\\ 1 & 28 & 9\\ -1 & 9 & 28\end{pmatrix}\]

Then $\mathrm{disc}(\Lambda) = 2738 = 2\cdot 37^2$ and it turns out that $\text{Cl}(\Lambda) = \{[\Lambda_1], [\Lambda_2]\}$ is of size two. Hence \[\mathcal{S}(O(V),O(\widehat{\Lambda}),\mathbb{C}) = \mathbb{C}\varphi_1\oplus \mathbb{C}\varphi_2\] where $\varphi_1,\varphi_2:\mathrm{Cl}(\Lambda)\rightarrow\mathbb{C}$ satisfy: \begin{align*}\varphi_1([\Lambda_1]) = 1 \qquad \varphi_1([\Lambda_2]) = 0\\ \varphi_2([\Lambda_1]) = 0\qquad \varphi_2([\Lambda_2]) = 1\end{align*}

The Hecke operators $T_p$ for $p\neq 37$ are represented by matrices (which are also the adjacency matrices for the $p$-neighbour graphs $G_p$):

\begin{align*}
&T_2 = \begin{pmatrix}
1 & 2\\ 1 & 2
\end{pmatrix}, 
T_3 = \begin{pmatrix}
2 & 2\\ 1 & 3
\end{pmatrix},
T_5 = \begin{pmatrix}
2 & 4\\ 2 & 4
\end{pmatrix},
T_7 = \begin{pmatrix}
2 & 6\\ 3 & 5
\end{pmatrix},\\
&T_{11} = \begin{pmatrix}
6 & 6\\ 3 & 9
\end{pmatrix}, T_{13} = \begin{pmatrix}
2 & 12\\ 6 & 8
\end{pmatrix},
T_{17} = \begin{pmatrix}
10 & 8\\ 4 & 14
\end{pmatrix},
T_{19} = \begin{pmatrix}
8 & 12\\ 6 & 14
\end{pmatrix},
\end{align*}

The simultaneous eigenvectors are $v_1 = [1,1]$ and $v_2 = [2,-1]$ (coordinates under the basis above). The system of eigenvalues corresponding to $v_1$ is: \begin{align*}\lambda_2 &= 3,\quad \lambda_3 = 4,\quad \lambda_5 = 6,\quad \lambda_7 = 8,\quad \lambda_{11} = 12,\quad \lambda_{13} = 14,\quad \lambda_{17} = 18,\quad \lambda_{19} = 20,\quad ...\end{align*} which we immediately recognise as $\lambda_p = p+1$, as expected. 

The system of eigenvalues corresponding to $v_2$ is  \begin{align*}\lambda_2 = 0,\quad \lambda_3 = 1,\quad \lambda_5 = 0,\quad \lambda_7 = -1,\quad \lambda_{11} = 3,\quad \lambda_{13} = -4,\quad \lambda_{17} = 6,\quad \lambda_{19} = 2,\quad ...\end{align*} which match the unique newform $f\in S_2^{\text{new}}(\Gamma_0(37),-1)$ having LMFDB label \url{37.2.a.b}: \begin{align*}f(\tau) &= q+q^3-2q^4-q^7-2q^9+3q^{11}-2q^{12}-4q^{13}+4q^{16}+6q^{17}+2q^{19}+...\end{align*}

Note that there exists a newform $f_2\in S_2^{\text{new}}(\Gamma_0(37),+1)$ with LMFDB label \url{37.2.a.a}. We do not see this at trivial weight, in accordance with Theorem \ref{2p2}. However, if we twist the representation by $\theta_{37}$ then we do see this newform, also in accordance with Theorem \ref{2p2} (since the sign vector $\varepsilon_{37}^{(37)}\varepsilon_{37}^{(37)} = (+1)$).

\end{example}

\begin{example}
We consider an example with non-trivial weight. Again, we refer to \cite{hein2016orthogonal} for detailed discussions of how these calculations are done in practice.

Consider the even lattice $\Lambda$ with Gram matrix:

\[T = \begin{pmatrix}2 & 1 & 0\\ 1 & 2 & 0\\ 0 & 0 & 6\end{pmatrix}.\] Then $\mathrm{disc}(\Lambda) = 18 = 2\cdot 3^2$ and it turns out that $\text{Cl}(\Lambda) = \{[\Lambda]\}$ is trivial. Hence \[\mathcal{S}(O(V),O(\widehat{\Lambda}),\mathbb{C}) = \{0\}\] and no modular data is recovered (agreeing with the fact that $S_2^{\text{new}}(\Gamma_0(3)) = \{0\}$).

When we use the weight $W_2$ we find that \[\mathcal{S}(O(V),O(\widehat{\Lambda}),W_2) = \{0\},\] which again agrees with the fact that $S_4^{\text{new}}(\Gamma_0(3)) = \{0\}$.

However, at weight $W_4$ we find that \[\mathcal{S}(O(V),O(\widehat{\Lambda}),W_4) = \mathbb{C}\varphi,\] for some explicit function $\varphi:\text{gen}(\Lambda) \rightarrow W_4$. Applying the Hecke operators $T_p$ for $p\neq 3$ gives eigenvalues: \begin{align*}\lambda_2 &= -6,\quad \lambda_5 = 6,\quad \lambda_7 = -40,\quad \lambda_{11} = -564,\quad \lambda_{13} = 638,\quad \lambda_{17} = 882,\quad \lambda_{19} = -556,\quad ...\end{align*} and we immediately find a match the unique newform $f\in S_6^{\text{new}}(\Gamma_0(3))$ having LMFDB label \url{3.6.a.a} (which has Atkin-Lehner sign $-1$): \begin{align*}f(\tau) &= q-6q^2+9q^3+4q^4+6q^5-54q^6-40q^7+168q^8+81q^9-36q^{10}-564q^{11}\\ &-36q^{12}+638q^{13}+240q^{14}+54q^{15}-1136q^{16}+882q^{17}-486q^{18}-556q^{19}+...\end{align*}

Finally, if the weight $W_4\otimes\theta_3$ is used then we find that: \[\mathcal{S}(O(V),O(\widehat{\Lambda}),W_4\otimes\theta_3) = \{0\}\] which matches the fact that $S_6^{\text{new}}(\Gamma_0(3),+1) = \{0\}$ (since the unique newform of weight $6$ and level $3$ has Atkin-Lehner sign $-1$).
\end{example}

\subsection{$N = pq$ semiprime}

We move on to the case $\mathrm{disc}(\Lambda) = 2p^2q^2$, since this illuminates a couple of additional features that will play a role in the general case. Recall that $p$ and $q$ are odd by assumption.

\begin{lem}
$\Lambda$ belongs to one of six genera. Up to swapping $p$ and $q$ these have local lattices:

\[\Lambda_t \sim \begin{cases}
\langle 1,1,2\rangle & \text{ if } t \nmid 2pq\\ \left\langle\begin{pmatrix}0 & 1\\ 1 & 0\end{pmatrix}, -2pq\right\rangle & \text{ if } t = 2\\ \langle r,p,pr\rangle & \text{ if } t = p \equiv 1 \bmod 8\\ \langle r,p,p\rangle & \text{ if } t = p \equiv 3 \bmod 8\\ \langle 1,p,pr\rangle & \text{ if } t = p \equiv 5 \bmod 8\\ \langle 1,p,p\rangle & \text{ if } t = p \equiv 7 \bmod 8\end{cases}\]

and

\vspace{0.1in}
\textbf{Genus A:}

\[\Lambda_q \sim \begin{cases}
\langle 1,q,q\rangle & \text{ if } q \equiv 1 \bmod 8\\ \langle 1,q,qr\rangle & \text{ if } q \equiv 3 \bmod 8\\ \langle r,q,q\rangle & \text{ if } q \equiv 5 \bmod 8\\ \langle r,q,qr\rangle & \text{ if } q \equiv 7 \bmod 8
\end{cases}\]

\textbf{Genus B:}

\[\Lambda_q \sim \begin{cases}
\langle 1,1,q^2\rangle & \text{ if } q \equiv 1 \bmod 8\\ \langle 1,r,q^2\rangle & \text{ if } q \equiv 3 \bmod 8\\ \langle 1,1,q^2r\rangle & \text{ if } q \equiv 5 \bmod 8\\ \langle 1,r,q^2r\rangle & \text{ if } q \equiv 7 \bmod 8
\end{cases}\]
\newpage
\textbf{Genus C:}

\[\Lambda_q \sim \begin{cases}
\langle 1,r,q^2r\rangle & \text{ if } q \equiv 1 \bmod 8\\ \langle 1,1,q^2r\rangle & \text{ if } q \equiv 3 \bmod 8\\ \langle 1,r,q^2\rangle & \text{ if } q \equiv 5 \bmod 8\\ \langle 1,1,q^2\rangle & \text{ if } q \equiv 7 \bmod 8
\end{cases}\]
\end{lem}

\begin{proof}
The first two isometries follow from Lemma \ref{localisom}.

Note that $S(V_{\infty}) = 1$ since $(V,Q)$ is positive definite. Since $S(V_t) = 1$ for all primes $t \nmid 2pq$ and $S(V_2) = -1$ we must have $S(V_p)S(V_q) = -1$ (by compatibility of Hasse invariants). We can swap $p$ and $q$ so that $S(V_p) = -1$ and $S(V_q) = 1$. Looking at Lemma \ref{localisom} we see that $\Lambda_p$ is uniquely determined up to isometry and $\Lambda_q$ takes one of three possibilities up to isometry.
\end{proof}

We can now prove transfer results for these genera.

\begin{thm}\label{2p2q2}
For each $k\geq 0$ and $d\,|\, pq$ there are the following Hecke-equivariant isomorphisms (for Hecke operators away from $pq$):

\[\mathcal{S}(O(V),O(\widehat{\Lambda}),W_{2k}\otimes\theta_d) \cong \begin{cases} S_{2k+2}^{p\mathrm{-new}}(\Gamma_0(pq),\varepsilon_p^{(pq)}\varepsilon_d^{(pq)})\\ S_{2k+2}^{p\mathrm{-new}}(\Gamma_0(pq^2),\varepsilon_p^{(pq^2)}\varepsilon_d^{(pq^2)})\\ S_{2k+2}^{\textup{new}}(\Gamma_0(p),\varepsilon_{p}^{(p)}\varepsilon_d^{(p)})\oplus S_{2k+2}^{\textup{new}}(\Gamma_0(pq^2),\varepsilon_{p}^{(pq^2)}\varepsilon_d^{(pq^2)})\end{cases}\] according to whether $\Lambda$ is in Genus $A$, Genus $B$ or Genus $C$, respectively.
\end{thm}

\begin{proof}
In all cases we have that $\Lambda_p \sim \langle a,p,pb\rangle$, so that we can apply $\hat{\psi}_{(p)}$ to $\Lambda$ to get a new even lattice $\Pi = V \cap \hat{\psi}_{(p)}(\Lambda)$ with $\mathrm{disc}(\Pi) = 2p^4q^2$ and satisfying $O(\widehat{\Lambda}) = O(\widehat{\Pi})$ (since $O(\Lambda_p) = O(\Pi_p)$ by Proposition \ref{psip} and $\Lambda_t = \Pi_t$ for all primes $t \neq p$).

Scaling the quadratic map allows us to view $\Pi$ as an even lattice in $(V,Q^{\frac{1}{p}})$ with $\mathrm{disc}(\Pi) = 2pq^2$ and local lattices:

\[\Pi_t \sim \begin{cases}\langle p^{-1}, p^{-1}, 2p^{-1}\rangle \sim \langle 1,1,2p\rangle & \text{ if } t \nmid 2pq\\ \left\langle \begin{pmatrix}0 & p^{-1}\\ p^{-1} & 0\end{pmatrix}, -2p^{-1}\right\rangle \sim \left\langle \begin{pmatrix}0 & 1\\ 1 & 0\end{pmatrix}, -2p\right\rangle & \text{ if } t = 2\\ \langle 1,r,pr\rangle & \text{ if } t = p\equiv 1 \bmod 8\\ \langle 1,1,pr\rangle & \text{ if } t=p\equiv 3 \bmod 8\\ \langle 1,r,p\rangle & \text{ if } t=p\equiv 5 \bmod 8\\ \langle 1,1,p\rangle  & \text{ if } t=p\equiv 7 \bmod 8 \end{cases}\] 

and

\vspace{0.1in}
\textbf{Genus A:}

\[\Lambda_q \sim \begin{cases}
\langle p^{-1},qp^{-1},qp^{-1}\rangle \sim \langle p,q,q\rangle & \text{ if } q \equiv 1 \bmod 8\\ \langle p^{-1},qp^{-1},qrp^{-1}\rangle \sim \langle p,q,qr\rangle & \text{ if } q \equiv 3 \bmod 8\\ \langle rp^{-1},qp^{-1},qp^{-1}\rangle \sim \langle rp, q, q \rangle & \text{ if } q \equiv 5 \bmod 8\\ \langle rp^{-1},qp^{-1},qrp^{-1}\rangle \sim \langle rp,q,qr \rangle & \text{ if } q \equiv 7 \bmod 8
\end{cases}\]

\newpage
\textbf{Genus B:}

\[\Lambda_q \sim \begin{cases}
\langle p^{-1},p^{-1},q^2p^{-1}\rangle \sim \langle 1,1,pq^2\rangle & \text{ if } q \equiv 1 \bmod 8\\ \langle p^{-1},rp^{-1},q^2p^{-1}\rangle \sim \langle 1,r,pq^2\rangle & \text{ if } q \equiv 3 \bmod 8\\ \langle p^{-1},p^{-1},q^2rp^{-1}\rangle \sim \langle 1,1,pq^2r\rangle & \text{ if } q \equiv 5 \bmod 8\\ \langle p^{-1},rp^{-1},q^2rp^{-1}\rangle \sim \langle 1,r,pq^2r\rangle & \text{ if } q \equiv 7 \bmod 8
\end{cases}\]

\textbf{Genus C:}

\[\Lambda_q \sim \begin{cases}
\langle p^{-1},rp^{-1},q^2rp^{-1}\rangle \sim\langle 1,r,pq^2r\rangle & \text{ if } q \equiv 1 \bmod 8\\ \langle p^{-1},p^{-1},q^2rp^{-1}\rangle\sim \langle 1,1,pq^2r\rangle & \text{ if } q \equiv 3 \bmod 8\\ \langle p^{-1},rp^{-1},q^2p^{-1}\rangle \sim \langle 1,r,pq^2\rangle & \text{ if } q \equiv 5 \bmod 8\\ \langle p^{-1},p^{-1},q^2p^{-1}\rangle \sim \langle 1,1,pq^2\rangle & \text{ if } q \equiv 7 \bmod 8
\end{cases}\]

By Lemma \ref{samespace} and Lemma \ref{scaling} we get an equality of Hecke modules: \[\mathcal{S}(O(V),O(\widehat{\Lambda}), W) = \mathcal{S}(O(V),O(\widehat{\Pi}), W)\] for any representation $W$ of $O(V)$.

If $\Lambda$ is in Genus B or Genus C then Lemma \ref{ram} tells us the following information. The quadratic space $(V,Q^{\frac{1}{p}})$ gives rise to a quaternion algebra $B = \mathrm{Cl}^0(V)$ of discriminant $p$. The order $\mathcal{O} = \mathrm{Cl}^0(\Pi)$ has discriminant $\mathrm{discrd}(\mathcal{O}) = pq^2$ and $\mathcal{O}_p$ is maximal (with $\varepsilon_p(\mathcal{O}) = -1$). For Genus B the order $\mathcal{O}_q$ is Eichler with $\varepsilon_q(\mathcal{O}) = 1$, whereas for Genus C it is non-Eichler with $\varepsilon_q(\mathcal{O}) = -1$. Then $\mathcal{O}$ is residually unramified everywhere and the relevant isomorphisms then follow from the results of Hein-Tornaria-Voight (see Theorem \ref{thm: 6.12 in HTV}).

If $\Lambda$ is in Genus A then $\Pi_q \sim \langle a,q,qb\rangle$ for some $a,b\in\{1,r\}$, so that we can apply $\hat{\psi}_{(q)}$ to $\Pi$ to get a new even lattice $\Omega = V \cap \hat{\psi}_{(q)}(\Pi)$ with $\mathrm{disc}(\Omega) = 2pq^4$ and satisfying $O(\widehat{\Pi}) = O(\widehat{\Omega})$ (since $O(\Pi_q) = O(\Omega_q)$ by Proposition \ref{psip} and $\Pi_t = \Omega_t$ for all primes $t \neq q$).

Scaling the quadratic map allows us to view $\Omega$ as an even lattice in $(V,Q^{(\frac{1}{pq})})$ with $\mathrm{disc}(\Sigma) = 2p$ and local lattices:

\[\Omega_t \sim \begin{cases}\langle q^{-1},q^{-1},2pq^{-1}\rangle \sim \langle 1,1,2pq\rangle & \text{ if } t \nmid 2pq\\ \left\langle \begin{pmatrix}0 & q^{-1}\\ q^{-1} & 0\end{pmatrix}, -2pq^{-1}\right\rangle \sim \left\langle \begin{pmatrix}0 & 1\\ 1 & 0\end{pmatrix}, -2pq\right\rangle & \text{ if } t = 2\\ \langle q^{-1},rq^{-1},prq^{-1}\rangle \sim \langle 1,r,pqr\rangle & \text{ if } t = p\equiv 1 \bmod 8\\ \langle q^{-1},q^{-1},prq^{-1}\rangle \sim \langle 1,1,pqr\rangle & \text{ if } t=p\equiv 3 \bmod 8\\ \langle q^{-1},rq^{-1},pq^{-1}\rangle\sim \langle 1,r,pq\rangle & \text{ if } t=p\equiv 5 \bmod 8\\ \langle q^{-1},q^{-1},pq^{-1}\rangle \sim\langle 1,1,pq\rangle & \text{ if } t=p\equiv 7 \bmod 8\\ \langle 1,1,pq \rangle & \text{ if } t=q \equiv 1 \bmod 8\\ \langle 1,r,pq\rangle & \text{ if } t=q \equiv 3 \bmod 8\\ \langle 1,1,pqr \rangle & \text{ if } t=q \equiv 5 \bmod 8\\ \langle 1,r,pqr \rangle & \text{ if } t=q \equiv 7 \bmod 8\end{cases}\] 

By Lemma \ref{samespace} and Lemma \ref{scaling} there is a further equality of Hecke modules: \[\mathcal{S}(O(V),O(\widehat{\Pi}),W) = \mathcal{S}(O(V),O(\widehat{\Omega}), W)\] for any representation $W$ of $O(V)$.

Lemma \ref{ram} tells us that the quadratic space $(V,Q^{(\frac{1}{pq})})$ still gives rise to a quaternion algebra $B = \mathrm{Cl}^0(V)$ of discriminant $p$, and that the order $\mathcal{O} = \mathrm{Cl}^0(\Omega)$ is Eichler with square-free discriminant $\mathrm{discrd}(\mathcal{O}) = pq$, hence residually unramified everywhere with $\varepsilon_p(\mathcal{O}) = -1$ and $\varepsilon_q(\mathcal{O}) = 1$. The result now follows from the results of Hein-Tornaria-Voight discussed earlier (see Theorem \ref{thm: 6.12 in HTV}).
\end{proof}

We consider three examples of trivial weight, once again using the Algebraic Modular Forms package in Magma. As usual, we refer to \cite{hein2016orthogonal} for detailed discussions of how these calculations are done in practice.

\begin{example}
Consider the even lattice $\Lambda$ with Gram matrix: \[T = \begin{pmatrix}
10 & -1 & -3\\ -1 & 10 & -3\\ -3 & -3 & 24
\end{pmatrix}\] This has $\mathrm{disc}(\Lambda) = 2178 = 2\cdot 3^2\cdot 11^2$ and $\Lambda$ is in Genus A with $p=11$ and $q = 3$ since \begin{align*}\Lambda_3 &\sim \langle 1, 3, 6\rangle\sim \langle 1,q,qr\rangle\\ \Lambda_{11} &\sim \langle 10,11,11 \rangle\sim \langle r,p,p\rangle.\end{align*} It turns out that $\text{Cl}(\Lambda) = \{[\Lambda_1],[\Lambda_2],[\Lambda_3]\}$ is of size three. Hence: \[\mathcal{S}(O(V),O(\widehat{\Lambda}),\mathbb{C}) = \mathbb{C}\varphi_1\oplus\mathbb{C}\varphi_2\oplus\mathbb{C}\varphi_3,\] where as before $\varphi_i([\Lambda_j]) = \delta_{i,j}$.

The Hecke operators $T_p$ for $p\neq 3,11$ are represented by matrices (which are also the adjacency matrices for the $p$-neighbour graphs $G_p$): \begin{align*}T_2 &= \begin{pmatrix}2 & 1 & 0\\ 2 & 0 & 1\\ 0 & 3 & 0\end{pmatrix},\quad T_5 = \begin{pmatrix}3 & 2 & 1\\ 4 & 2 & 0\\ 6 & 0 & 0\end{pmatrix},\quad T_7 = \begin{pmatrix}6 & 2 & 0\\ 4 & 2 & 2\\ 0 & 6 & 2\end{pmatrix},\\ \quad T_{13} &= \begin{pmatrix}8 & 4 & 2\\ 8 & 6 & 0\\ 12 & 0 & 2\end{pmatrix},\quad T_{17} = \begin{pmatrix}10 & 6 & 2\\ 12 & 4 & 2\\ 12 & 6 & 0\end{pmatrix},\quad T_{19} = \begin{pmatrix}12 & 6 & 2\\ 12 & 6 & 2\\ 12 & 6 & 2\end{pmatrix},\quad...\end{align*}

The simultaneous eigenvectors are $v_1 = [1,1,1]$, $v_2 = [1,-4,6]$ and $v_3 = [1,-1,-3]$ (coordinates under the basis above). The system of eigenvalues corresponding to $v_1$ are $\lambda_p = p+1$, as expected. 

The system of eigenvalues corresponding to $v_2$ are: \[\lambda_2 = -2,\quad \lambda_5 = 1,\quad \lambda_7 = -2,\quad \lambda_{13} = 4,\quad \lambda_{17} = -2,\quad \lambda_{19} = 0,\quad ...\] which match the unique newform $f_1\in S_2^{\text{new}}(\Gamma_0(11))$ having LMFDB label \url{11.2.a.a} (which has Atkin-Lehner sign $-1$): \begin{align*}f_1(\tau) &= q-2q^2-q^3+2q^4+q^5+2q^6-2q^7-2q^9-2q^{10}+q^{11}\\ &-2q^{12}+4q^{13}+4q^{14}-q^{15}-4q^{16}-2q^{17}+4q^{18}+2q^{20}+...\end{align*}

The system of eigenvalues corresponding to $v_3$ are: \[\lambda_2 = 1,\quad \lambda_5 = -2,\quad \lambda_7 = 4,\quad \lambda_{13} = -2, \quad \lambda_{17} = -2,\quad \lambda_{19} = 0,\quad ...\] which match the unique newform $f_2\in S_2^{\text{new}}(\Gamma_0(33))$ having LMFDB label \url{33.2.a.a} (which has Atkin-Lehner sign $+1$ at $3$ and $-1$ at $11$): \begin{align*}f_2(\tau) &= q+q^2-q^3-q^4-2q^5-q^6+4q^7-3q^8+q^{9}-2q^{10}+q^{11}\\ &+q^{12}-2q^{13}+4q^{14}+2q^{15}-q^{16}-2q^{17}+q^{18}+2q^{20}+...\end{align*}

Note that all newforms of levels $11$ and $33$ have Atkin-Lehner eigenvalues $+1$ at $3$ and $-1$ at $11$. However, if there had been newforms with other Atkin-Lehner eigenvalues, then we would not be able to see them in this calculation. We would have to twist the weight by $\theta_d$ for some suitable $d$ (as per Theorem \ref{2p2q2}).

\end{example}

\begin{example}
Now consider the even lattice $\Lambda$ with Gram matrix: \[T = \begin{pmatrix}
2 & 1 & 0\\ 1 & 50 & 0\\ 0 & 0 & 22
\end{pmatrix}\] This also has $\mathrm{disc}(\Lambda) = 2178 = 2\cdot 3^2\cdot 11^2$ and $\Lambda$ is in Genus B with $p=11$ and $q = 3$ since: \begin{align*}\Lambda_3 &\sim \langle 1,2,9 \rangle \sim \langle 1,r,q^2\rangle\\ \Lambda_{11} &\sim \langle 10,11,11 \rangle\sim \langle r,p,p\rangle.\end{align*} It turns out that $\text{Cl}(\Lambda) = \{[\Lambda_1],[\Lambda_2],[\Lambda_3],[\Lambda_4],[\Lambda_5]\}$ is of size five. Hence: \[\mathcal{S}(O(V),O(\widehat{\Lambda}),\mathbb{C}) = \mathbb{C}\varphi_1\oplus\mathbb{C}\varphi_2\oplus\mathbb{C}\varphi_3\oplus\mathbb{C}\varphi_4\oplus\mathbb{C}\varphi_5\] where as before $\varphi_i([\Lambda_j]) = \delta_{i,j}$.

The Hecke operators $T_p$ for $p\neq 3,11$ are represented by matrices (which are also the adjacency matrices for the $p$-neighbour graphs $G_p$): \begin{align*}T_2 &= \begin{pmatrix}0 & 1 & 2 & 0 & 0\\ 1 & 0 & 0 & 2 & 0\\ 1 & 0 & 0 & 0 & 2\\ 0 & 1 & 0 & 0 & 2\\ 0 & 0 & 1 & 1 & 1\end{pmatrix},\quad T_5 = \begin{pmatrix}2 & 2 & 0 & 2 & 0\\ 2 & 2 & 2 & 0 & 0\\ 0 & 1 & 1 & 2 & 2\\ 1 & 0 & 2 & 1 & 2\\ 0 & 0 & 1 & 1 & 4\end{pmatrix},\quad T_7 = \begin{pmatrix}0 & 0 & 4 & 0 & 4\\ 0 & 0 & 0 & 4 & 4\\ 2 & 0 & 2 & 0 & 4\\ 0 & 2 & 0 & 2 & 4\\ 1 & 1 & 2 & 2 & 2\end{pmatrix},\\ \quad T_{13} &= \begin{pmatrix}2 & 0 & 0 & 4 & 8\\ 0 & 2 & 4 & 0 & 8\\ 0 & 2 & 4 & 4 & 4\\ 2 & 0 & 4 & 4 & 4\\ 2 & 2 & 2 & 2 & 6\end{pmatrix},\quad T_{17} =\begin{pmatrix}0 & 2 & 4 & 4 & 8\\ 2 & 0 & 4 & 4 & 8\\ 2 & 2 & 2 & 4 & 8\\ 2 & 2 & 4 & 2 & 8\\ 2 & 2 & 4 & 4 & 6\end{pmatrix},\quad T_{19} = \begin{pmatrix}0 & 0 & 4 & 4 & 12\\ 0 & 0 & 4 & 4 & 12\\ 2 & 2 & 4 & 4 & 8\\ 2 & 2 & 4 & 4 & 8\\ 3 & 3 & 4 & 4 & 6\end{pmatrix},\quad...\end{align*}

The simultaneous eigenvectors are (coordinates under the basis above): \begin{align*}v_1 &= [1,1,1,1,1]\\  v_2 &=[2,0,-2,-1,1]\\ v_3 &= [0,2,-1,-2,1]\\ v_4 &= [2,2,0,0,-1]\\  v_5 &=[1,-1,1,-1,0]\end{align*} The system of eigenvalues corresponding to $v_1$ is $\lambda_p = p+1$ as expected. 

The system of eigenvalues corresponding to $v_2$ and $v_3$ are equal and are: \[\lambda_2 = -2,\quad \lambda_5 = 1,\quad \lambda_7 = -2,\quad \lambda_{13} = 4,\quad \lambda_{17} = -2,\quad \lambda_{19} = 0,\quad ...\] which match the newform $f_1\in S_2^{\text{new}}(\Gamma_0(11))$ having LMFDB label \url{11.2.a.a} (and Atkin-Lehner sign $+1$ at $3$ and $-1$ at $11$): \begin{align*}f_1(\tau) &= q - 2q^2 - q^3 + 2q^4 + q^5 + 2q^6 - 2q^7 - 2q^9 - 2q^{10} + q^{11}\\ &- 2q^{12} + 4q^{13} + 4q^{14} - q^{15} - 4q^{16} - 2q^{17} + 4q^{18} + 2q^{20}+...\end{align*} 

The system of eigenvalues corresponding to $v_4$ are: \[\lambda_2 = 1,\quad \lambda_5 = -2,\quad \lambda_7 = 4,\quad \lambda_{13} = -2,\quad \lambda_{17} = -2,\quad \lambda_{19} = 0,\quad ...\] which match the newform $f_2\in S_2^{\text{new}}(\Gamma_0(33))$ having LMFDB label \url{33.2.a.a} (and Atkin-Lehner sign $+1$ at $3$ and $-1$ at $11$): \begin{align*}f_2(\tau) &= q + q^2 - q^3 - q^4 - 2q^5 - q^6 + 4q^7 - 3q^8 + q^9 - 2q^{10}+ q^{11}\\ &+ q^{12} - 2q^{13} + 4q^{14} + 2q^{15} - q^{16} - 2q^{17} + q^{18} + 2q^{20}+...\end{align*} 

The system of eigenvalues corresponding to $v_5$ are: \[\lambda_2 = 1,\quad \lambda_5 = 4,\quad \lambda_7 = -2,\quad \lambda_{13} = -2,\quad \lambda_{17} = -2,\quad \lambda_{19} = -6,\quad ...\] which match the newform $f_3\in S_2^{\text{new}}(\Gamma_0(99),+1,-1)$ having LMFDB label \url{99.2.a.c}: \begin{align*}f_3(\tau) &= q + q^2 - q^4 + 4q^5 - 2q^7 - 3q^8 + 4q^{10}+ q^{11}\\ &- 2q^{13} - 2q^{14} - q^{16} - 2q^{17} - 6q^{19} - 4q^{20}+...\end{align*}

Note that we do not see the newforms of level $99$ with LMFDB labels \url{99.2.a.a}, \url{99.2.a.b} and \url{99.2.a.d} in this computation. This is because they have the wrong Atkin-Lehner signs (but we would see them by twisting the weight by $\theta_d$ for a suitable $d$, as in Theorem \ref{2p2q2}).

Given that the image of the transfer in this case is $S_{2k+2}^{11\mathrm{-new}}(\Gamma_0(99),+1,-1)$, one might have been expecting to see the newform of level $11$ three times and the newform of level $33$ twice (since there are three distinct lifts from level $11$ to level $99$ and two distinct lifts from level $33$ to level $99$ under the degeneracy maps). However, note that these lifts will have varying Atkin-Lehner eigenvalues, accounting for the fact that we only see this newform with multiplicity two.
\end{example}

\begin{example}
Finally, consider the even lattice $\Lambda$ with Gram matrix: \[T = \begin{pmatrix}
10 & 1 & 0\\ 1 & 10 & 0\\ 0 & 0 & 22
\end{pmatrix}\] This also has $\mathrm{disc}(\Lambda) = 2178 = 2\cdot 3^2\cdot 11^2$ and $\Lambda$ is in Genus C with $p=11$ and $q=3$ since: \begin{align*}\Lambda_3 &\sim \langle 1,1,18 \rangle\sim \langle 1,1,q^2r\rangle\\ \Lambda_{11} &\sim \langle 10,11,11 \rangle\sim \langle r,p,p\rangle.\end{align*} It turns out that $\text{Cl}(\Lambda) = \{[\Lambda_1],[\Lambda_2],[\Lambda_3]\}$ is of size three. Hence: \[\mathcal{S}(O(V),O(\widehat{\Lambda}),\mathbb{C}) = \mathbb{C}\varphi_1\oplus\mathbb{C}\varphi_2\oplus\mathbb{C}\varphi_3,\] where as before $\varphi_i([\Lambda_j]) = \delta_{i,j}$.

The Hecke operators $T_p$ for $p\neq 3,11$ are represented by matrices (which are also the adjacency matrices for the $p$-neighbour graphs $G_p$): \vspace{0.2in}\begin{align*}T_2 &= \begin{pmatrix}1 & 2 & 0\\ 1 & 0 & 2\\ 0 & 2 & 1\end{pmatrix},\quad T_5 = \begin{pmatrix}4 & 2 & 0\\ 1 & 3 & 2\\ 0 & 2 & 4\end{pmatrix},\quad T_7 = \begin{pmatrix}0 & 4 & 4\\ 2 & 2 & 4\\ 2 & 4 & 2\end{pmatrix},\\ \quad T_{13} &= \begin{pmatrix}2 & 4 & 8\\  2 & 8 & 4\\ 4 & 4 & 6\end{pmatrix},\quad T_{17} = \begin{pmatrix}2 & 8 & 8\\ 4 & 6 & 8\\ 4 & 8 & 6\end{pmatrix},\quad T_{19} = \begin{pmatrix}0 & 8 & 12\\ 4 & 8 & 8\\ 6 & 8 & 6\end{pmatrix},\quad...\end{align*}
\newpage
The simultaneous eigenvectors are $v_1 = [1,1,1]$, $v_2 = [2,-3,2]$ and $v_3 = [2,0,-1]$. The system of eigenvalues corresponding to $v_1$ are $\lambda_p = p+1$, as expected. 

The system of eigenvalues corresponding to $v_2$ are \[\lambda_2 = -2,\quad \lambda_5 = 1,\quad \lambda_7 = -2,\quad \lambda_{13} = 4,\quad \lambda_{17} = -2,\quad \lambda_{19} = 0,\quad ...\] which match the unique newform $f_1\in S_2^{\text{new}}(\Gamma_0(11))$ having LMFDB label \url{11.2.a.a} (which has Atkin-Lehner sign $-1$): \begin{align*}f_1(\tau) &= q-2q^2-q^3+2q^4+q^5+2q^6-2q^7-2q^9-2q^{10}+q^{11}\\ &-2q^{12}+4q^{13}+4q^{14}-q^{15}-4q^{16}-2q^{17}+4q^{18}+2q^{20}+...\end{align*} 

The system of eigenvalues corresponding to $v_3$ are:  \[\lambda_2 = 1,\quad \lambda_5 = 4,\quad \lambda_7 = -2,\quad \lambda_{13} = -2,\quad \lambda_{17} = -2,\quad \lambda_{19} = -6,\quad ...\] which match the newform $f_3\in S_2^{\text{new}}(\Gamma_0(99),+1,-1)$ having LMFDB label \url{99.2.a.c}: \begin{align*}f_3(\tau) &= q + q^2 - q^4 + 4q^5 - 2q^7 - 3q^8 + 4q^{10}+ q^{11}\\ &- 2q^{13} - 2q^{14} - q^{16} - 2q^{17} - 6q^{19} - 4q^{20}+...\end{align*}

Note that again we do not see the newforms of level $99$ with LMFDB labels \url{99.2.a.a}, \url{99.2.a.b} and \url{99.2.a.d}, due to these having the wrong Atkin-Lehner signs (but we would see them by twisting the weight). We also do not see the newform $f_2$ of level $33$ with LMFDB label \url{33.2.a.a}, despite having the correct Atkin-Lehner eigenvalues. Also the newform $f_1$ of level $11$ only appears once. These observations demonstrate the difference between Genus B and Genus C, as explained by Theorem \ref{2p2q2}.

\end{example}

\subsection{$\mathbf{N}$ square-free}\label{N1N2N3N4}

We’re now ready to consider the general case of $\mathrm{disc}(\Lambda) = 2N^2$ with $N$ odd and square-free.

Note that there is no positive definite ternary even lattice of discriminant $2$ since as we saw earlier we require $S(V_p) = -1$ for at least one odd prime (to balance the fact that $S(V_2) = -1$). It follows that $N>1$, so that $N$ has at least one prime divisor. 

Since we already anticipate the need to sort the primes dividing $N$ according to Hasse invariants and local lattice structures, we will write $N = N_1N_2N_3N_4$ where:

\begin{itemize}
\item{$N_1$ is the product of primes $p \,|\, N$ such that $S(V_p) = -1$,}
\item{$N_2$ is the product of primes $p \,|\, N$ such that $S(V_p) = 1$ and $\Lambda_p \sim \langle a,p,pb\rangle$ for some $a,b\in \{1,r\}$,}
\item{$N_3$ is the product of primes $p \,|\, N$ such that $S(V_p) = 1$,  $\Lambda_p\sim \langle 1,a,p^2b\rangle$ for some $a,b\in \{1,r\}$ with $\left(\frac{-a}{p}\right) = 1$}
\item{$N_4$ is the product of primes $p \,|\, N$ such that $S(V_p) = 1$,  $\Lambda_p\sim \langle 1,a,p^2b\rangle$ for some $a,b\in \{1,r\}$ with $\left(\frac{-a}{p}\right) = -1$}
\end{itemize}

As usual, we classify the possible genera first.

\begin{lem}
Suppose that $N$ has $n$ prime divisors. Then $\Lambda$ belongs to one of $2^{n-1}(2^n-1)$ genera. Sorting the primes dividing $N$ according to the decomposition $N = N_1N_2N_3N_4$ above, the local lattices are:

\[\Lambda_t\sim\begin{cases}
\langle 1,1,2\rangle & \text{ if } t\nmid 2N\\ \left\langle\begin{pmatrix}0 & 1\\ 1 & 0\end{pmatrix}, -2\right\rangle & \text{ if } t = 2\\ \langle r,t,tr\rangle & \text{ if } t \,|\, N_1 \text{ and } t \equiv 1 \bmod 8\\ \langle r,t,t\rangle & \text{ if } t \,|\, N_1 \text{ and } t \equiv 3 \bmod 8\\ \langle 1,t,tr\rangle & \text{ if } t \,|\, N_1 \text{ and } t \equiv 5 \bmod 8\\ \langle 1,t,t\rangle & \text{ if } t \,|\, N_1 \text{ and } t \equiv 7 \bmod 8\\ \langle 1,t,t\rangle & \text{ if } t \,|\, N_2 \text{ and } t \equiv 1 \bmod 8\\ \langle 1,t,tr\rangle & \text{ if } t \,|\, N_2 \text{ and } t \equiv 3 \bmod 8\\ \langle r,t,t\rangle & \text{ if } t \,|\, N_2 \text{ and } t \equiv 5 \bmod 8\\ \langle r,t,tr\rangle & \text{ if } t \,|\, N_2 \text{ and } t \equiv 7\bmod 8\\ \langle 1,1,t^2\rangle & \text{ if } t \,|\, N_3 \text{ and } t\equiv 1 \bmod 8\\ \langle 1,r,t^2\rangle & \text{ if } t \,|\, N_3 \text{ and } t\equiv 3\bmod 8\\ \langle 1,1,t^2r\rangle & \text{ if } t \,|\, N_3 \text{ and } t\equiv 5 \bmod 8\\ \langle 1,r,t^2r\rangle & \text{ if } t \,|\, N_3 \text{ and } t\equiv 7 \bmod 8\\ \langle 1,r,t^2r\rangle & \text{ if } t \,|\, N_4 \text{ and } t\equiv 1 \bmod 8\\ \langle 1,1,t^2r\rangle & \text{ if } t \,|\, N_4 \text{ and } t\equiv 3\bmod 8\\ \langle 1,r,t^2\rangle & \text{ if } t \,|\, N_4 \text{ and } t\equiv 5 \bmod 8\\ \langle 1,1,t^2\rangle & \text{ if } t \,|\, N_4 \text{ and } t\equiv 7 \bmod 8
\end{cases}\]
\end{lem}

\begin{proof}
Once again, the first two isometries follow from Lemma \ref{localisom}.

Note that $S(V_{\infty}) = 1$ since $(V,Q)$ is positive definite. Since $S(V_t) = 1$ for all primes $t \nmid 2N$ and $S(V_2) = -1$ we must have $\prod_{p \,|\, N}S(V_p) = -1$ (by compatibility of Hasse invariants). Thus an odd number of primes must satisfy $S(V_p) = -1$ (which we can arrange as above to be the primes dividing $N_1$). Looking at Lemma \ref{localisom} we see that $\Lambda_t$ is uniquely determined up to isometry for primes $t| N_1$ and $\Lambda_t$ takes one of three possibilities up to isometry when $t|\frac{N}{N_1}$.

Letting the number of primes dividing $N_1$ be $2g+1$ we see that the number of genera is then \[S = \sum_{\substack{m\leq n\\  m \text{ odd }}}3^{n-m} = 2^{n-1}(2^n-1)\] by the Binomial Theorem.
\end{proof}

We are now ready to prove Theorem \ref{main}, concerning the explicit transfer for positive definite, ternary, even lattices of discriminant $2N^2$ with $N$ odd and square-free.

\begin{customthm}{1.1}\label{2N2}
For each $k\geq 0$ and $d \,|\, N$ there exists a Hecke-equivariant isomorphism (for Hecke operators away from $N$): \[\mathcal{S}(O(V),O(\widehat{\Lambda}),W_{2k}\otimes \theta_d) \cong \bigoplus_{M_4|N_4}S_{2k+2}^{N_1M_4^2\mathrm{-new}}(\Gamma_0(N_1N_2N_3^2M_4^2), \varepsilon_{N_1}^{(N_1N_2N_3^2M_4^2)}\varepsilon_d^{(N_1N_2N_3^2M_4^2)})\]
\end{customthm}

\begin{proof}
The proof is very similar to the previous results. One can apply each of the maps $\hat{\psi}_{(p)}$ for $p \,|\, N_1N_2$ and scale the quadratic map by $\frac{1}{N_1N_2}$ to produce an even lattice $\Pi$ of discriminant $2N_1N_2N_3^2N_4^2$ in the quadratic space $(V,Q^{(\frac{1}{N_1N_2})})$ with local lattices:

\[\Pi_t\sim\begin{cases}
\langle (N_1N_2)^{-1},(N_1N_2)^{-1},2(N_1N_2)^{-1}\rangle \sim \langle 1,1,2N_1N_2\rangle & \text{ if } t\nmid 2N\\ \left\langle\begin{pmatrix}0 & (N_1N_2)^{-1}\\ (N_1N_2)^{-1} & 0\end{pmatrix}, -2(N_1N_2)^{-1}\right\rangle\sim\left\langle\begin{pmatrix}0 & 1\\ 1 & 0\end{pmatrix}, -2N_1N_2\right\rangle & \text{ if } t = 2\\ \langle t(N_1N_2)^{-1},tr(N_1N_2)^{-1},t^2r(N_1N_2)^{-1}\rangle \sim \langle 1,r,rN_1N_2\rangle & \text{ if } t \,|\, N_1 \text{ and } t \equiv 1 \bmod 8\\ \langle t(N_1N_2)^{-1},t(N_1N_2)^{-1}, t^2r(N_1N_2)^{-1}\rangle \sim \langle 1,1,rN_1N_2\rangle & \text{ if } t \,|\, N_1 \text{ and } t \equiv 3 \bmod 8\\ \langle t(N_1N_2)^{-1},tr(N_1N_2)^{-1},t^2(N_1N_2)^{-1}\rangle \sim \langle 1,r,N_1N_2\rangle & \text{ if } t \,|\, N_1 \text{ and } t \equiv 5 \bmod 8\\ \langle t(N_1N_2)^{-1},t(N_1N_2)^{-1},t^2(N_1N_2)^{-1}\rangle \sim \langle 1,1,N_1N_2\rangle & \text{ if } t \,|\, N_1 \text{ and } t \equiv 7 \bmod 8\\ \langle t(N_1N_2)^{-1},t(N_1N_2)^{-1},t^2(N_1N_2)^{-1}\rangle \sim\langle 1,1,N_1N_2\rangle & \text{ if } t \,|\, N_2 \text{ and } t \equiv 1 \bmod 8\\ \langle t(N_1N_2)^{-1},tr(N_1N_2)^{-1}, t^2(N_1N_2)^{-1}\rangle \sim\langle 1,r,N_1N_2\rangle & \text{ if } t \,|\, N_2 \text{ and } t \equiv 3 \bmod 8\\ \langle t(N_1N_2)^{-1},t(N_1N_2)^{-1},t^2r(N_1N_2)^{-1}\rangle \sim \langle 1,1,rN_1N_2\rangle & \text{ if } t \,|\, N_2 \text{ and } t \equiv 5 \bmod 8\\ \langle t(N_1N_2)^{-1},tr(N_1N_2)^{-1}, t^2r(N_1N_2)^{-1}\rangle\sim \langle 1,r,rN_1N_2\rangle & \text{ if } t \,|\, N_2 \text{ and } t \equiv 7\bmod 8\\ \langle (N_1N_2)^{-1},(N_1N_2)^{-1},t^2(N_1N_2)^{-1}\rangle \sim \langle 1,1,t^2N_1N_2\rangle & \text{ if } t \,|\, N_3 \text{ and } t\equiv 1 \bmod 8\\ \langle (N_1N_2)^{-1},r(N_1N_2)^{-1},t^2(N_1N_2)^{-1}\rangle\sim \langle 1,r,t^2N_1N_2\rangle & \text{ if } t \,|\, N_3 \text{ and } t\equiv 3\bmod 8\\ \langle (N_1N_2)^{-1},(N_1N_2)^{-1},t^2r(N_1N_2)^{-1}\rangle\sim\langle 1,1,t^2rN_1N_2\rangle & \text{ if } t \,|\, N_3 \text{ and } t\equiv 5 \bmod 8\\ \langle (N_1N_2)^{-1},r(N_1N_2)^{-1},t^2r(N_1N_2)^{-1}\rangle\sim\langle 1,r,t^2rN_1N_2\rangle & \text{ if } t \,|\, N_3 \text{ and } t\equiv 7 \bmod 8\\ \langle (N_1N_2)^{-1},r(N_1N_2)^{-1},t^2r(N_1N_2)^{-1}\rangle\sim\langle 1,r,t^2rN_1N_2\rangle & \text{ if } t \,|\, N_4 \text{ and } t\equiv 1 \bmod 8\\ \langle (N_1N_2)^{-1},(N_1N_2)^{-1},t^2r(N_1N_2)^{-1}\rangle\sim\langle 1,1,t^2rN_1N_2\rangle & \text{ if } t \,|\, N_4 \text{ and } t\equiv 3\bmod 8\\ \langle (N_1N_2)^{-1},r(N_1N_2)^{-1},t^2(N_1N_2)^{-1}\rangle\sim\langle 1,r,t^2N_1N_2\rangle & \text{ if } t \,|\, N_4 \text{ and } t\equiv 5 \bmod 8\\ \langle (N_1N_2)^{-1},(N_1N_2)^{-1},t^2(N_1N_2)^{-1}\rangle\sim\langle 1,1,t^2N_1N_2\rangle & \text{ if } t \,|\, N_4 \text{ and } t\equiv 7 \bmod 8
\end{cases}\]

By Lemma \ref{samespace} and Lemma \ref{scaling} we get an equality of Hecke modules: \[\mathcal{S}(O(V),O(\widehat{\Lambda}), W) = \mathcal{S}(O(V),O(\widehat{\Pi}), W)\] for any representation $W$ of $O(V)$.

Lemma \ref{ram} tells us that the quadratic space $(V,Q^{(\frac{1}{N_1N_2})})$ gives rise to a quaternion algebra $B = \mathrm{Cl}^0(V)$ of discriminant $N_1$, and that the order $\mathcal{O} = \mathrm{Cl}^0(\Pi)$ has local orders $\mathcal{O}_t$ that are:

\begin{itemize}
\item{maximal of discriminant $t$ for each prime $t \,|\, N_1$ (with $\varepsilon_t(\mathcal{O}) = -1$),}
\item{Eichler of discriminant $t$ for each prime $t \,|\, N_2$ (with $\varepsilon_t(\mathcal{O}) = 1$),}
\item{Eichler of discriminant $t^2$ for each prime $t \,|\, N_3$ (with $\varepsilon_t(\mathcal{O}) = 1$),}
\item{non-Eichler of discriminant $t^2$ for each prime $t \,|\, N_4$ (with $\varepsilon_t(\mathcal{O}) = -1$),}
\end{itemize}

In particular, $\mathcal{O}$ is residually unramified everywhere, and so the result follows from the results of Hein-Tornaria-Voight (see Theorem \ref{thm: 6.12 in HTV}), noting that \[(D,N, N^+, N^-) = (N_1, N_1N_2N_3^2N_4^2, N_2N_3^2, N_1N_4^2)\] and that the valid divisors $M|N^-$ are those of the form $M = N_1M_4^2$ for some $M_4 \,|\,N_4$.
\end{proof}

Theorem \ref{2N2} completes the task that we set ourselves at the beginning of the paper, i.e.\ to give a full explicit transfer for any even lattice $\Lambda$ in a positive definite ternary quadratic space $(V,Q)$ over $\mathbb{Q}$ with $\mathrm{disc}(\Lambda) = 2N^2$ for $N$ odd and square-free. An interesting question is to consider what the analogous results say when $N$ is even and square-free. Such results will hopefully feature in a future paper of the authors.

\end{document}